\documentclass[11pt,leqno]{article}

\usepackage[left=2.9cm,right=2.9cm,top=2.9cm,bottom=2.9cm]{geometry}

\usepackage{todonotes}
\usepackage{eufrak}
\usepackage{hyperref}
\usepackage{pb-diagram}
\usepackage{amsthm}
\usepackage{amsmath}
\usepackage{txfonts}
\usepackage{graphicx}
\newcommand{\footnoteremember}[2]{
  \footnote{#2}
  \newcounter{#1}
  \setcounter{#1}{\value{footnote}}
}
\newcommand{\footnoterecall}[1]{
  \footnotemark[\value{#1}]
}

 \newtheorem{thm}{Theorem}[section]
 \newtheorem{cor}[thm]{Corollary}
 \newtheorem{lem}[thm]{Lemma}
 \newtheorem{prop}[thm]{Proposition}
 \newtheorem{conj}[thm]{Conjecture}
 \theoremstyle{definition}
 \newtheorem{defn}[thm]{Definition}
 \theoremstyle{remark}
 \newtheorem{rem}[thm]{Remark}
 
 \numberwithin{equation}{section}
\newcommand{\kommentar}[1]{}

\let \nc \newcommand
\let \rnc \renewcommand

 \rnc \phi \varphi \rnc \epsilon \varepsilon

\nc {\R}{{\bf R}} \nc {\C}{{\bf C}} \nc {\Z}{{\bf Z}}\nc {\Q}{{\bf
Q}}

\nc {\bd}{\begin{description}} \nc {\ed}{\end{description}} \nc {\bi}{\begin{itemize}} \nc {\ei}{\end{itemize}}
\nc {\be}{\begin{enumerate}} \nc {\ee}{\end{enumerate}} \nc {\bdm}{\begin{displaymath}} \nc
{\edm}{\end{displaymath}} \nc {\bea}{\begin{eqnarray*}} \nc {\eea}{\end{eqnarray*}} \nc {\baa}{\begin{alignat*}}
\nc {\eaa}{\end{alignat*}} \nc {\bsp}{\begin{split}} \nc {\esp}{\end{split}} \nc {\beq}{\begin{equation}} \nc
{\eeq}{\end{equation}} \nc {\btab}{\begin{tabular}} \nc {\etab}{\end{tabular}} \nc {\ba}{\begin{array}} \nc
{\ea}{\end{array}}

\newcommand{\cA}{{\mathcal A}}

\newcommand{\cN}{{\mathcal N}}
\newcommand{\cM}{{\mathcal M}}
\newcommand{\cG}{{\mathcal G}}

\newcommand{\cO}{{\mathcal O}}
\newcommand{\cL}{{\mathcal L}}

\newcommand{\cH}{{\mathcal H}}

\newcommand{\cB}{{\mathcal B}}
\newcommand{\fg}{\EuFrak{g}}

\newcommand{\ft}{\EuFrak{t}}
\newcommand{\fsu}{\EuFrak{su}}
\newcommand{\fsl}{\EuFrak{sl}}
\newcommand{\fu}{\EuFrak{u}}

\def\co{\colon\thinspace}

\nc{\LOT}[1]{L^2(\Omega^{0+1+2}(T;#1))} \nc{\LO}[2]{L^2(\Omega^{0+1+2}(#1;#2))} \nc{\HT}[1]{{H^{0+1+2}(T;#1)}}
\nc{\cHT}[1]{{\cH^{0+1+2}(T;#1)}} \nc{\stor}{D^2\times S^1}

\nc{\Om}{\Omega} \nc{\dist}{{\rm dist}}

\nc{\homeo}{\approx} \nc{\im}{\operatorname{\rm
im}}\nc{\diag}{\operatorname{\rm diag}} \nc{\Crit}{\operatorname{\rm
Crit}}\nc{\grad}{\operatorname{\rm grad}}
\nc{\Hess}{\operatorname{\rm Hess}}\nc{\Ad}{\operatorname{\rm Ad}}
\rnc{\O}{\operatorname{\rm O}}\nc{\U}{\operatorname{\rm
U}}\nc{\SO}{\operatorname{\rm SO}} \nc{\SU}{\operatorname{\rm SU}}\nc{\Spin}{\operatorname{\rm Spin}}
\nc{\Cl}{\operatorname{\rm Cl}} \nc{\PD}{\operatorname{\rm PD}}
\nc{\ad}{\operatorname{\rm ad}} \nc{\vol}{\operatorname{\rm vol}}
\newcommand{\rk}{\operatorname{\rm rk}}\nc{\hol}{\operatorname{\rm
hol}} \nc{\re}{\operatorname{\rm Re}} \nc{\Id}{\operatorname{\rm
Id}} \nc{\Mas}{\operatorname{\rm Mas}} \nc{\Stab}{\operatorname{\rm
Stab}} \nc{\SF}{\operatorname{\rm SF}} \rnc{\ker}{\operatorname{\rm
ker}} \nc{\tr}{\operatorname{\rm tr}} \nc{\sign}{\operatorname{\rm
sign}} \nc{\eq}{\operatorname{\rm
eq}} \nc{\alg}{\operatorname{\rm
alg}}\nc{\topo}{\operatorname{\rm
top}} \nc{\Spec}{\operatorname{\rm
Spec}}\nc{\coker}{\operatorname{\rm coker}}
\rnc{\hom}{\operatorname{\rm Hom}} \nc{\ch}{\operatorname{\rm ch}}
\nc{\End}{\operatorname{\rm End}} \nc{\Aut}{\operatorname{\rm Aut}}
\nc{\lat}{(\frac{1}{2}\Z)^2} \nc{\hz}{\tfrac{1}{2}\Z}
\nc{\la}{\langle} \nc{\ra}{\rangle} \nc {\Ra}{\Rightarrow} \nc
{\La}{\Leftarrow} \nc {\lla}{\longleftarrow} \nc
{\nach}{\rightarrow} \nc {\equ}{\Leftrightarrow} \nc
{\gdw}{\Longleftrightarrow} \nc {\lra}{\longrightarrow} \nc
{\Lra}{\Longrightarrow} \nc {\lmt}{\longmapsto} \nc
{\quer}{\overline} \nc {\tensor}{\otimes} \nc {\Rt}{\widetilde\R}
\nc {\contract}{\lrcorner} \nc{\sgn}{\operatorname{\rm sgn}}
\newcommand{\kom}[1]{}

\nc{\CS}{\operatorname{\rm CS}} \nc{\Ch}{\operatorname{\rm Ch}} \nc{\Td}{\operatorname{\rm
Td}}\nc{\Rank}{\operatorname{\rm Rank}} \nc{\GL}{\operatorname{\rm GL}} 
\nc{\proj}{\operatorname{\rm proj}}

\begin{document}

\title{The Witten-Reshetikhin-Turaev invariants of \\ finite order mapping tori II}

\author{J{\o}rgen Ellegaard Andersen\footnoteremember{Thanks}{Supported in part by the center of excellence grant "Center for quantum geometry of Moduli Spaces" from the Danish National Research Foundation.} \and Benjamin Himpel\footnoterecall{Thanks}}




\date{}




\maketitle

\begin{abstract}We identify the leading order term of the asymptotic expansion of the Witten-Reshetikhin-
Turaev invariants for finite order mapping tori with classical invariants for all simple and simply-connected
compact Lie groups. The square root of the Reidemeister torsion is used as a density
on the moduli space of flat connections and the leading order term is identified with the integral over this moduli space of this density weighted by a certain phase for each component of the moduli space. We also
identify this phase in terms of classical invariants such as Chern-Simons invariants, eta invariants, spectral flow and the rho invariant. As a result, we show agreement with the semiclassical
approximation as predicted by the method of stationary phase.\end{abstract}

\section{Introduction}

The Witten-Reshetikhin-Turaev quantum invariants were first proposed
by Witten in his seminal paper \cite{witten89}, where he studied the
Chern-Simons quantum field theory for a simple, simply connected
compact Lie group $G$. He did so using path integral techniques,
which let him to propose a combinatorial surgery formula for the
invariants.

Shortly thereafter Reshetikhin and Turaev gave a rigorous
construction of these quantum invariants using the representation
theory of quantum groups.  In fact, they subsequently constructed
the whole topological quantum field theory (TQFT) $Z^{(k)}_G$ in
\cite{reshetikhin-turaev90, reshetikhin-turaev91,
turaev2010_QuantumInvariants} for $G=\SU(2)$. The other classical
groups were treated in \cite{turaev-wenzl1993_QuantumInvariants,
turaev-wenzl1997_LinkInvariants}. The TQFT for $G=\SU(2)$ were also
constructed using skein theory by Blanchet, Habegger, Masbaum and
Vogel in \cite{blanchet-habegger-masbaum-vogel92,
blanchet-habegger-masbaum-vogel95}. Since then these constructions
have been extended to other Lie groups $G$ through the effort of
many people. For a complete list we refer to the references in
\cite{turaev2010_QuantumInvariants}.

Witten also analyzed the Chern-Simons path integral from a
perturbative point of view. The identification of the leading order
asymptotics of the invariants in terms of classical topological
invariants in the case of an isolated, irreducible flat connection,
was proposed by Witten in \cite{witten89}. There has been subsequent
proposals for refinements and generalizations to this, for example
by Freed and Gompf \cite[Equation (1.3)]{freed-gompf91} and by
Jeffrey \cite[Equation (5.1)]{jeffrey92}, partially supported by
computations of the quantum invariants, as well as solely from path
integral techniques Axelrod, Lawrence, Mari{\~ n}o, Rozansky,
Singer, Zagier \cite{axelrod-singer92, axelrod-singer94, rozansky95,
rozansky1996_ResidueFormulasSeifertManifolds,
lawrence-rozansky1999_WRTinvariantsSeifert,
lawrence-zagier1999_QuantumInvariants}.

Both the perturbative expansion conjecture \cite[Conjecture 7.6]{andersen2002:TheAsymptoticExpansionConjecture} and the asymptotic expansion conjecture \cite[Conjecture 7.7]{andersen2002:TheAsymptoticExpansionConjecture} address the asymptotic behavior of the quantum invariants which we expect from the perturbative point of view. The first conjecture attempts to give a detailed description of the asymptotic expansion in terms of an integral of certain classical topological invariants over the moduli space of flat connections and Feynman diagrams, which come from stationary phase approximation and the perturbative expansion respectively. Let us refer to the part of this conjecture which is concerned with the leading order term as the {\em the semiclassical approximation conjecture}. Since the statement of the perturbative expansion conjecture requires some interpretation, the conjecture has been reduced to the mathematically precise asymptotic expansion conjecture \cite[Conjecture 7.7]{andersen2002:TheAsymptoticExpansionConjecture}. This work is part of a series of papers analyzing the asymptotic behavior of the quantum invariants in the case of finite order mapping tori. While the first part \cite{andersen95} focuses solely on the asymptotic expansion conjecture, we establish the semiclassical approximation conjecture by starting from the results in \cite{andersen95}.

The asymptotic expansion of the Witten-Reshetikhin-Turaev quantum invariants has been studied by a number of authors, and various results have been obtained for
certain classes of closed three manifolds \cite{freed-gompf91,garoufalidis92_Thesis,jeffrey92,rozansky95,lawrence-rozansky1999_WRTinvariantsSeifert, lawrence-zagier1999_QuantumInvariants,rozansky1996_ResidueFormulasSeifertManifolds,
marino2005:CSTheoryPert,hansen2001,hansen2005,hansen-takata2002,hansen-takata2004,beasley-witten2005:NonAbelianLocalization, charles-marche2011:KnotAsymptoticsI,charles-marche2011:KnotAsymptoticsII,charles2011:TorusKnotAsymptotics,charles2010:AsymptoticsMappingTorus,marche-paul2011:ToeplitzSkein}. An overview can be found in the introduction of \cite{andersen95}. Let us mention the ones, which take the extra step of expressing the terms in the asymptotic expansion of the quantum invariants geometrically. Freed and Gompf \cite{freed-gompf91} considered lens spaces and certain Brieskorn spheres, and they used computer calculations to confirm the semiclassical approximation conjecture. In a subsequent paper, Lisa Jeffrey \cite{jeffrey92} formally proved this conjecture for lens spaces as well as mapping tori of genus one surfaces with restrictions either on the choice of monodromy map or the structure group. A few missing details about the spectral flow contribution are formulated as \cite[Conjecture 5.8]{jeffrey92}, which has later been partially confirmed \cite{kirk-klassen94,himpel2005}. Garoufalidis did work similar to \cite{jeffrey92} in his thesis \cite{garoufalidis92_Thesis}. In particular, he found formulas for the quantum invariants for certain Seifert manifolds and rewrote them so it was obvious that they satisfy the asymptotic expansion conjecture. Rozansky \cite{rozansky95} studied the
asymptotic expansion of the $\SU(2)$ quantum invariants for Seifert manifolds with non-zero orbifold Euler characteristic. In particular, he confirms the perturbative expansion conjecture up to the 2-loop contributions. Rozansky \cite{rozansky1996_ResidueFormulasSeifertManifolds} also studied general Seifert manifolds in the case $\SU(2)$ and expresses the contributions to the asymptotic expansion of the quantum invariants in terms of intersection pairings on the moduli space, but it is unclear to what extent his calculations are rigorous. Nevertheless, his formulas bear a formal resemblance to the ones in \cite{andersen95}. Beasley and Witten \cite{beasley-witten2005:NonAbelianLocalization} considered the path integral formula for these quantum invariants for Seifert manifolds with non-zero orbifold Euler characteristic \cite[Equation (3.20)]{beasley-witten2005:NonAbelianLocalization}. Since they are working with path integrals, their work is per se not rigorous, however they provide path integral arguments for the fact that the
perturbation expansion of these invariants are finite (modulo the framing correction
term). For a mathematical proof of that result please consult the paper \cite{andersen95}. Recently, Charles and March\'e \cite{charles2011:TorusKnotAsymptotics,charles-marche2011:KnotAsymptoticsII} proved the semiclassical approximation conjecture for Dehn fillings of torus knots and the figure eight knot for $\SU(2)$.

In contrast, our results are for finite order mapping tori of surfaces of genus greater than one, which are Seifert manifolds with vanishing orbifold Euler characteristic. Therefore our family of mapping tori is disjoint from the families considered by Jeffrey and Rozansky. Furthermore, we would like to point out that their approach relies on explicit computations of the quantum invariants and the Poisson resummation trick, while we identify the emerging spectral invariants on a more conceptual level based on geometric quantization. Note that it has recently been confirmed in a series of papers \cite{andersen-ueno2007, andersen-ueno2007b, andersen-ueno2006, andersen-ueno2011_TQFT-Construction}, that the gauge theory construction of the quantum invariants for $G=\SU(n)$ coincides with the combinatorial or equivalently skein theory construction.

Since we have explicit combinatorial expressions  for the quantum
invariants, it is sensible to extract the perturbation expansion
from these exact formulas. To this end we need an ansatz for the
kind of asymptotic expansion we can expect based on Witten's path
integral formula for the invariants. This leads us to the asymptotic
expansion conjecture \cite[Conjecture 7.7]{andersen2002:TheAsymptoticExpansionConjecture},
\cite{andersen95} and \cite[Conjecture 1]{andersen-hansen2006}.
\begin{conj}[Asymptotic expansion conjecture]\label{ConjAsymptoticExpansion} Let $X$ be a closed 3--manifold. There exist constants
(depending on X) $d_j \in \frac12 \Z$ and $b_j \in \C$ for
$j=0,\ldots, n$ and $a^l_j \in \C$ for $j=0,1,\ldots,n$, $l
=1,2,\ldots$ such that the asymptotic expansion of $Z_G^{(k)}(X)$ in
the limit $k\to\infty$ is given by
\begin{equation}\label{AsymptoticExpansion}
Z_G^{(k)}(X) \sim \sum_{j=0}^n e^{2\pi i k q_j}
k^{d_j}b_j\left(1+\sum_{j=1}^\infty a_j^lk^{-l/2}\right),
\end{equation}
where $q_0=0, q_1,\ldots,q_n$ are finitely many different values of
the Chern-Simons functional on the space of flat $G$--connections on
$X$.
\end{conj}

Here {\bf $\sim$} denotes {\em asymptotic equivalence} in the
Poincar\'{e} sense, which means the following: Let
\[d = \max\{d_0,\ldots,d_n\}.\]
Then for any non-negative integer $L$, there is a $c_L \in \R$
such that
\[\left| Z^{(k)}_G(X) - \sum_{j=0}^n e^{2\pi i k q_j} k^{d_j}
b_j \left( 1 + \sum_{l=0}^L a_j^l k^{-l/2}\right) \right| \leq c_L
k^{d-(L+1)/2}\] for all levels $k$. Of course such a condition only
puts limits on the large $k$ behavior of $Z^{(k)}_G(X)$.

Through the previous definition we can make the following definition of the {\em leading order term} of the asymptotics.

\begin{defn} If $Z_G^{(k)}(X)$ satisfies Conjecture
\ref{ConjAsymptoticExpansion}, then we write
\[
Z_G^{(k)}(X) \mbox{ } \dot\sim \mbox{ } \sum_{j=0}^n e^{2\pi i k q_j} k^{d_j}b_j
\]
and we call the sum on the right {\em the leading order term} of
(the asymptotic expansion of) $Z_G^{(k)}(X)$.
\end{defn}

It is this leading order term for which there conjecturally is a classical topological expression. In fact, let $\cM(X)$ be the moduli space flat $G$-connections on $X$ and let us write the component decomposition as
$$\cM(X) = \bigcup_{c\in C_X} \cM(X)_c.$$
One expects that a square root of the Reidemeister torsion produces
a measure on $\cM(X)$ \cite{jeffrey92, rozansky95, jeffrey-mclellan2011:NonabelianLocalization,mclellan2010:NonabelianLocalization}. Combining results from
\cite{axelrod-singer92, axelrod-singer94,
rozansky95,rozansky95_LoopExpansion,
rozansky1996_ResidueFormulasSeifertManifolds,axelrod97,
lawrence-rozansky1999_WRTinvariantsSeifert,
lawrence-zagier1999_QuantumInvariants} and the references therein we
arrive at the following conjectured formula for the leading order
term.

\begin{conj}[Semiclassical approximation conjecture] The leading order term of $Z_G^{(k)}(X)$ with respect to the Atiyah
2--framing \cite{atiyah1990:Framings} is given by
\begin{equation}\label{stationaryphaseSF}
\begin{split}
Z_{G}^{(k)}(X)\mbox{ } \dot\sim \mbox{ }\sum_{c\in C_X} \frac{1}{|Z(G)|}e^{ \pi i\dim
G(1+b^1(X))/4} \int_{A\in \cM(X)_{c}} &\sqrt{\tau_X(A)}
e^{2\pi i \CS_X(A)(k+h)}\\& e^{2\pi i
\left(\SF(\theta,A)/4-(\dim(H^0(X,d_A))+\dim(H^1(X,d_A)))/8\right)}k^{d_c}
\end{split}
\end{equation}
and
$$d_c =  \frac{1}{2} \max_{A \in\cM(X)_{c}}
            \left( \dim(H^1(X,d_A)) - \dim(H^0(X,d_A)) \right),
$$
where $\max$ here means the maximum value $\dim(H^1(X,d_A)) -
\dim(H^0(X,d_A)) $ attained on a Zariski open subset of $
\cM(X)_{c}$.
\end{conj}

Note that the exact solution of the path integral depends on a framing of twice the tangent bundle as a $\Spin(6)$ bundle \cite[Equations (2.24) and (2.25)]{witten89}. The dependence on the 2--framing explained in the skein theoretic definition in \cite{blanchet-habegger-masbaum-vogel95} and in more general setting of quantum invariants for general modular categories inTuraev's exposition \cite{turaev2010_QuantumInvariants}. See also the discussion on this point by Freed and Gompf in \cite{freed-gompf91}.

In Appendix \ref{Heuristics} we review the heuristics by which the method of
stationary phase applies to the Chern-Simons path integral and
produces this conjecture. Note, that Appendix \ref{Heuristics} contains the only
non-rigorous part of this paper, which we decided to keep for
motivational purposes. In the conjectured formula
\eqref{stationaryphaseSF} we see expressions for the constants $b_j$
and $d_j$ in terms of Reidemeister torsion, spectral flow and
dimensions of twisted cohomology groups.

In this paper we consider the Witten-Reshetikhin-Turaev quantum
invariants of finite order mapping tori. Let $\Sigma$ be a closed oriented surface. Then
the mapping torus $\Sigma_f$ of a diffeomorphism $f\co
\Sigma \to \Sigma$ is defined to be
\[
\Sigma_f = \Sigma \times I / (x,1) \sim (f(x),0)
\]
with the orientation given by the product orientation with the
standard orientation on the interval $I=[0,1]$. Let $\cM(\Sigma)$ be the moduli space of flat $G$ connections on $\Sigma$. This is a stratified symplectic space on which the mapping class group acts. We assume that $f$
is of finite order, and denote by
$|\cM(\Sigma)|\subset \cM(\Sigma)$ the fixed point set of $f^*$. Denote by $C$ an indexing set
for the set of all connected components $\{|\cM(\Sigma)|_c\}_{c\in
C}$ of $|\cM(\Sigma)|$. There is a map $r\co \cM(\Sigma_f) \to
|\cM(\Sigma)|$ given by restricting a flat connection on $\Sigma_f$
to $\Sigma \times \{0\}$. We will write $\cM(\Sigma_f)_c =
r^{-1}(|\cM(\Sigma)|_c)$. Let the prime superscript denote the part
which is irreducible in $\cM(\Sigma)$: $\cM(\Sigma)'\subset
\cM(\Sigma)$ denotes the irreducible subset, while $\cM(\Sigma_f)'_c
= r^{-1}(|\cM(\Sigma)|'_c)$.

Choose a complex structure $\sigma$ on $\Sigma$, which is fixed by $f$. Consider the moduli space $\cM_\sigma$ of semi-stable $G^\C$
bundles over $\Sigma_\sigma$. We identify $\cM_\sigma$ and $\cM(\Sigma)$ as stratified symplectic spaces, but $\cM_\sigma$ has the additional structure of a normal projective variety. We write
\[|{\mathcal M}_\sigma| = \bigcup_{c\in C} |{\mathcal M}_\sigma|_c\] for the component decomposition of the fixed point set of $f$.
Following the notation of \cite{baum-fulton-macpherson1979_Riemann-Roch} and \cite{baum-fulton-quart_LefschetzRiemann-Roch} we denote the Grothendieck group of all equivariant coherent sheaves on ${\mathcal M}_\sigma$ by
$K^{\eq}_0({\mathcal M}_\sigma)$ and the Grothendieck group of all coherent sheaves on $|{\mathcal M}_\sigma|_c$ by $K^{\alg}_0(|{\mathcal M}_\sigma|_c)$. Let
\[L_\bullet^c \co K^{\eq}_0({\mathcal M}_\sigma) \rightarrow K^{\alg}_0(|{\mathcal
M}_\sigma|_c)\otimes {\C}\]
be the localizing homomorphism defined in \cite[\S 2]{baum-fulton-quart_LefschetzRiemann-Roch}, and
\[\tau_\bullet \co K^{\alg}_0(|{\mathcal M}_\sigma|_c) \rightarrow H_\bullet(|{\mathcal
M}_\sigma|_c)\]
the homomorphism defined in the theorem on page 180 in \cite{baum-fulton-macpherson1979_Riemann-Roch}. The Lefschetz-Riemann-Roch
formula of Baum, Fulton, McPherson and Quart then states that
\[\tr(f \co H^0({\mathcal M}_\sigma, {\mathcal L}^k) \rightarrow
H^0({\mathcal M}_\sigma, {\mathcal L}^k)) =
\sum_{c\in C} a_c^k \ch({\mathcal L}^k|_{|{\mathcal M}_\sigma|_c}) \cap
\tau_\bullet(L_\bullet^c({\mathcal O}_{{\mathcal M}_\sigma}))\]
where $a_c$ is the
complex number by which $f$ acts on ${\mathcal L}|_{|{\mathcal M}_\sigma|_c}.$ For the convenience of the reader we review the Lefschetz-Riemann-Roch theorem for singular varieties in Appendix \ref{LRR}. From \cite[Theorem 2.19]{freed95} we see that $f$ acts on $\cL_{[A]}$ by multiplication with $\exp(2\pi i \CS_{\Sigma_f}(A))$ and that $\CS_{\Sigma_f}(A) \mod \Z$ is constant for $A\in \cM(\Sigma_f)_c
$. If we write $ \exp(2\pi i \CS_{\Sigma_f}(c)) =\exp(2\pi i \CS_{\Sigma_f}(A))$ for $A \in
\cM(\Sigma_f)_c$, we get
\[ a_c = \exp(2\pi i \CS_{\Sigma_f}(c)).\]
Clearly
\[\ch({\mathcal L}^k|_{|{\mathcal M}_\sigma|_c})
= \exp(k\omega_c),\] where $\omega_c$ is the
the restriction of $c_1({\mathcal L})$ to $|{\mathcal M}_\sigma|_c$. In summary, we get the following theorem of
\cite{andersen95}, which proves the asymptotic expansion conjecture for finite order mapping tori.

\begin{thm}[{\cite[Theorem 8.2]{andersen95}}]\label{ThmWRTInvariant} The Witten-Reshetikhin-Turaev invariants of $\Sigma_f$ are given by
\begin{equation}\label{WRTInvariant}
Z^{(k)}_G({\Sigma_f}) = \det(f)^{-\frac{1}{2}\zeta} \sum_{c\in C}
\exp(2\pi i k \CS_{\Sigma_f}(c)) \exp(k\omega_c) \cap \tau_\bullet
(L^c_\bullet(\cO_{\cM(\Sigma)})),
\end{equation}
where
\[
\zeta = \frac{k \dim G}{k+h}
\]
is the central charge of the theory, $h$ is the dual Coxeter number
of $G$, $\det(f)^{-\frac{1}{2}\zeta} $ is the framing correction defined in Section \ref{LeadingTerms}.
\end{thm}

Note that the restriction of $c_1(\cL)$ to the smooth part of the moduli space $\cM_\sigma$ can be represented by the K\"ahler form on $\cM_\sigma$. The evaluation of the top power of the class $c_1(\cL^k)$ on  $\tau_\bullet (L^c_\bullet(\cO_{\cM(\Sigma)}))$ is just the integration of this top form over the smooth part of $|\cM(\Sigma)|_c$, when this component has the property that it has an open dense part of irreducibles. This follows from the lemma on page 129 of \cite{baum-fulton-macpherson1975:RiemannRoch} and part (6) of the Riemann-Roch Theorem of \cite{fulton-gillet1983:RiemannRoch} (see also Appendix B below). Starting with Theorem \ref{ThmWRTInvariant}, our main result is the following theorem, which
applies to all finite order elements $f$ of the mapping class group
of $\Sigma$.

\begin{thm}\label{ThmIdentification} For each $c\in C$ such that $\cM(\Sigma_f)'_c$ is
nonempty we have
\begin{equation}\begin{split}\label{EqnIrreducibleContribution}
k^{-d_c} \det(f)^{-\frac{1}{2}\zeta} e^{2\pi i k \CS_{\Sigma_f}(c)}
&\frac{1}{d_c!} (\exp(k\omega_c) \cap
\tau_\bullet(L_\bullet^c(\cO_{\cM(\Sigma)})))\\
&=\frac{1}{|Z(G)|} \int_{A\in\cM(\Sigma_f)'_c} e^{2\pi i k
\CS_{\Sigma_f}(A)} \sqrt{\tau_{\Sigma_f}(A)}  e^{2\pi
i\frac{\rho_A(\Sigma_f)}{8}} + O(\frac{1}{k}),
\end{split}\end{equation}
where $\rho_A(\Sigma_f)$ is the classical $rho$-invariant.
\end{thm}

In Section \ref{SpectralFlow} we give a proof of a well-known
formula relating the spectral flow, the $\rho$--invariant and the
Chern-Simons invariant for an arbitrary Lie group $G$, since the
only proof we have found in the literature is for $\SU(2)$ (see
\cite[Section 7]{kirk-klassen-ruberman94}). The precise
relation---stated in Theorem \ref{SFRhoCS}---shows in particular,
that Theorem \ref{ThmIdentification} has an equivalent formulation
in terms of spectral flow, which has the following theorem as an
immediate consequence, once combined with Theorem \ref{ThmWRTInvariant}.

\begin{thm} \label{ThmStationarySF} If $\cM(\Sigma)'_c$ is nonempty for every $c\in C$, then the above conjecture for the
leading order term is correct, i.e.
\begin{equation}\label{EqnStationarySF}
\begin{split}
Z^{(k)}_{G}(\Sigma_f) \mbox{ } \dot\sim \mbox{ }\sum_{c\in C} \frac{1}{|Z(G)|}e^{ \pi i\dim
G(1+b^1(\Sigma_f))/4} \int_{A\in \cM(\Sigma_f)_{c}'} &\sqrt{\tau_{\Sigma_f}(A)}
e^{2\pi i \CS_{Ã}(A)(k+h)}\\& e^{2\pi i
\left(\SF(\theta,A)/4-(\dim(H^0(\Sigma_f,d_A))+\dim(H^1(\Sigma_f,d_A)))/8\right)}k^{d_c}
\end{split}
\end{equation}
and
$$d_c =  \frac{1}{2} \max_{A \in\cM(\Sigma_f)_{c}}
            \left( \dim(H^1(\Sigma_f,d_A)) - \dim(H^0(\Sigma_f,d_A)) \right),
$$
where $\max$ here means the maximum value $\dim(H^1(\Sigma_f,d_A)) -
\dim(H^0(\Sigma_f,d_A)) $ attained on a Zariski open subset of $
\cM(\Sigma_f)_{c}$.
\end{thm}

Recall that $\cM(\Sigma_f)_c'$ consists of the irreducible representation whose restriction to $\Sigma$ is irreducible. By Theorem \ref{ThmValidity} the hypothesis of Theorem \ref{ThmStationarySF} is satisfied in the case $G = \SU(n)$ and $g(\Sigma/\langle f\rangle)>1$. Note that unlike for lens spaces, the stationary phase approximation is in general not exact: for example, for $f = \Id$ lower order terms in the asymptotic expansion do not in general vanish, as one easily sees, since the Todd class of the moduli spaces are in general none trivial.

This paper is organized as follows. Section \ref{CS&MS} contains a preliminary discussion about the Chern-Simons functional and the moduli spaces of flat connections. In Section \ref{LeadingTerms} we
express the leading order term of the Witten-Reshetikhin-Turaev
invariants for each $c\in C$ as certain integrals of differential
geometric data. In Section \ref{ReidemeisterTorsion} and
\ref{ReidemeisterTorsionMappingTori} we review Reidemeister torsion
and compute it for mapping tori. In Section \ref{RhoInvariant} we
review the $\rho$--invariant and an essential result for finite
order mapping tori by Bohn \cite{bohn2009}. Section
\ref{identification} combines the main results from Sections
\ref{LeadingTerms} and \ref{ReidemeisterTorsionMappingTori} to
identify the classical invariants in the leading order term of the
$Z_G^{(k)}(X)$ in the limit $k\to \infty$. Section
\ref{SpectralFlow} gives an equivalent formulation of this
identification in Section \ref{LeadingTerms} in terms of spectral
flow. In Appendix \ref{Heuristics} we present the heuristics which lead to the
conjectured identification of the leading order term with classical
topological invariants. In Appendix \ref{LRR} we review the Lefschetz-Riemann-Roch theorem for singular varieties.

The results of this paper relies on the results of \cite{andersen95}, which were obtained by using the gauge theory approach to the Witten-Reshetikhin-Turaev TQFT. The first named author has obtained other results about this TQFT using the gauge theory approach, such as the asymptotic faithfulness of the quantum representations \cite{andersen2006_Asymptoticfaithfulness} and the determination of the Nielsen-Thurston classification via these same representations \cite{andersen2008_NielsenThurstonClassification} (see also \cite{andersen-masbaum-ueno2006}). He has further related these quantum representations to deformation quantization of moduli spaces both in the abelian and in the non-abelian case, please see \cite{andersen2005}, \cite{andersen-gammelgaard2010} and \cite{andersen2011_HitchinsConnection}. The second named author has answered some open questions by Jeffrey \cite{jeffrey92} about this TQFT for torus-bundles over $S^1$ by using cut-and-paste methods to perform spectral flow computations \cite{himpel2005}.

We would like to thank Henning Haahr Andersen, Hans Boden, Jens Carsten Jantzen, Johan Dupont, and Nicolai Reshetikhin for helpful discussions.

\section{The Chern-Simons invariant and moduli spaces of flat connections}
\label{CS&MS}
In this section we give some necessary definitions and make some remarks
 regarding normalizations before we consider the moduli space and recall its decomposition into connected components.

\subsection*{Normalizations for the Chern-Simons functional and Poincare duality}

Let $\la \cdot ,\cdot \ra$ be a multiple of the Killing form on the
Lie Algebra $\fg$ of a simple and simply-connected compact Lie group
$G$ normalized so that $-\frac{1}{6}\la \theta \wedge
[\theta\wedge\theta]\ra$ is a minimal integral generator of
$H^3(G,\R)$, where $\theta$ is the Maurer-Cartan form. A connection
on a principal $G$--bundle $P$ is a $G$--equivariant, horizontal Lie
algebra valued 1--form on $P$. The group of gauge transformations
$\cG$ consists of all bundle automorphism $P \to P$, which acts on
connections by pull-back. Let $X$ be an oriented, closed
$3$--manifold. Since $G$ is simply-connected, every principal
$G$--bundle over $X$ is trivializable; therefore let us fix a
trivialization to simplify notation, which allows us to identify the
affine space of connections with Lie algebra valued 1-forms $\cA_X =
\Omega^1(X;\fg)$. Furthermore, the moduli space of flat
$G$--connections on $X$, denoted by $\cM(X)$, can be identified with
the moduli space of flat connections in the trivial $G$--bundle. The
Chern-Simons invariant is the map $\cA_X \to \R$ given by
\begin{equation}\label{CS}
\CS_X(A) = \int_X \la A \wedge dA + \frac{1}{3} A \wedge [A\wedge
A]\ra.
\end{equation}
It is not difficult to see that---with our choice of normalization
for the inner product on $\fg$---$\CS_X$ factors through $\cG$ as an
$\R/\Z$--valued map. It is also not difficult to see, that the map
$\cG \to \Z$ given by $\Phi \mapsto \CS_X(\Phi^*A)- \CS_X(A)$ is
onto.

Let $\Sigma$ be a closed oriented surface and consider the space of
connections $\cA_\Sigma$ in a trivial principal $G$--bundle over
$\Sigma$. The symplectic structure on $\cA_\Sigma$ is naturally
given by
\begin{equation}\label{SymplecticForm}
\omega(a,b) = - 2\int_\Sigma \la a\wedge b\ra.
\end{equation}
This gives a (stratified) symplectic structure on the moduli space
$\cM(\Sigma)$ of flat $G$--connections on $\Sigma$.

In order to view the square root of Reidemeister torsion as a
density, we need to identify $H^2(\Sigma_f,d_A)$ with
$(H^1(\Sigma_f,d_A))^*$ using Poincar\'e duality, which depends on a choice of inner product on $\fg$. For $\fg$--valued differential forms $a$ and $b$ we set
\[\PD(a)(b) = \int_{\Sigma_f} 2 \la a\wedge b\ra,\]
which descends to the Poincar\'e duality isomorphism $H^k(\Sigma_f,d_A) \to (H^{3-k}(\Sigma_f,d_A))^*$. Note that the factor $2$ might seem unnatural, but as we mention in Appendix \ref{Heuristics}, there is a choice involved, and the correct choice is the one which satisfies $\PD(a)(b) = -\omega(a,b)$.

\subsection*{Connected components of $\cM(\Sigma_f)$ and $|\cM(\Sigma)|$}

Recall that $r\co \cM(\Sigma_f)'_c \to |\cM(\Sigma)|'_c$ is a
$|Z(G)|$--sheeted covering map (see \cite[Section 7]{andersen95}). We
get a complete description of the leading order term of the
Witten-Reshetikhin-Turaev invariants in terms of a sum of integrals
over the components $\cM(\Sigma_f)'_c$ of $\cM'(\Sigma_f)$, if every $|\cM(\Sigma)|_c$ contains an irreducible representation.  The connected components of $\cM(\Sigma)$ have been studied by Goldman \cite{goldman1988:ComponentsReps}. The components of the fixed point set $|\cM(\Sigma)|$ of $f$ are analyzed in \cite[Section 6]{andersen95}. In this section we will see, in which situations all $\cM(\Sigma_f)'_c$ are nonempty.

For a chosen diffeomorphism $f\co \Sigma \to \Sigma$ of order $m$
consider the projection $\pi\co \Sigma \to \tilde\Sigma$ to the
quotient surface $\tilde\Sigma \coloneqq \Sigma/\la f\ra$. $\Sigma$ is an
$m$--fold branched cover over $\tilde\Sigma$ with branch points
$\tilde p_1,\ldots,\tilde p_n$, for which \[\pi^{-1}(\tilde p_i)=
\{p_i , f(p_i),\ldots, f^{m_i-1}(p_i)\}\quad \text{with} \quad m_i<m.\] Choose small
disjoint closed discs $D_i$ around each $p_i$, $i=1,\ldots,n$, such
that $f^j(D_i)$, $j=0,\ldots,m_i$, $i=0,\ldots,n$, are disjoint. Let
$\Sigma'$ be the complement of the interior of all these discs and
$\tilde\Sigma'\coloneqq \Sigma'/\la f\ra$.

\begin{figure}[ht]
\begin{center}
\includegraphics{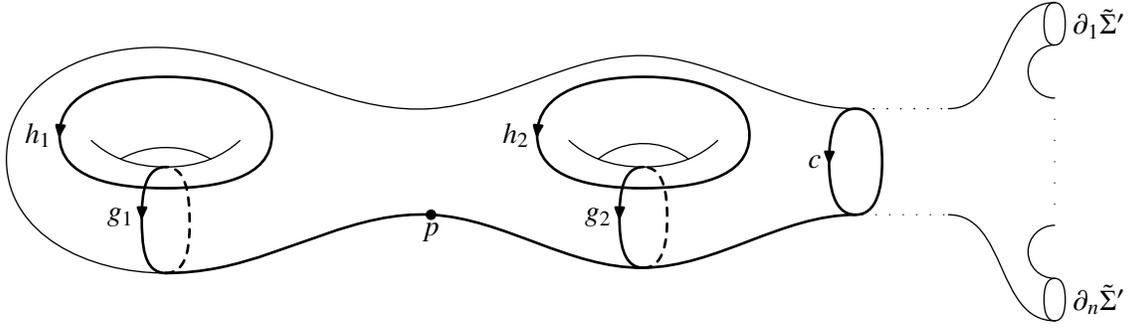}
\caption{Images of the representation $\rho\co \pi_1(\tilde\Sigma')\to G$}\label{FigSurface1}
\end{center}
\end{figure}

The indexing set for the connected components of $|\cM'(\Sigma)|$ is
shown in \cite[Proposition 6.1]{andersen95} to be
\[
\Delta \coloneqq \{ (z,c_1,\ldots,c_n) \in Z(G) \times \Cl^n \mid
z\in c_i^{l_i}\}/Z(G),
\]
where $\Cl$ is the set of conjugacy classes of $G$ and $l_i
\coloneqq \frac{m}{m_i}$. We have a surjective map $\Delta \to C$ if $|\cM(\Sigma)| = \overline{|\cM'(\Sigma)|}$. By \cite[Proposition 6.3]{andersen95} this is the case for $G=\SU(n)$. For $c(\delta)\coloneqq
(c_1^{-k_1},\ldots,c_n^{-k_n})$ let $\cM(\tilde\Sigma',c(\delta))$ be
the moduli space of flat $G$--connections on $\tilde\Sigma'$ with
holonomy around $\partial_i\tilde \Sigma'$ in $c_i^{k_i}$,
$i=1,\ldots,n$. By \cite[Theorem 6.1]{andersen95}, a component $|\cM'(\Sigma)|_\delta$ can be
described as the space $\cM'' (\tilde\Sigma',c(\delta))/Z_\delta$,
where $\cM''(\tilde\Sigma',c(\delta))$ consists  of the flat
$G$--connections in $\cM'(\tilde\Sigma',c(\delta))$, which remain
irreducible when pulled back via $\pi$.

\begin{prop}\label{PropIrreduciblesExist} Let $G$ be a connected compact Lie group, $\tilde \Sigma'$
a genus two surface with one boundary circle $\partial_i \tilde\Sigma'$,
and $\pi\co \Sigma' \to \tilde\Sigma'$ a covering map. Then
$\cM''(\tilde\Sigma',c)$ is nonempty for every $c\in \Cl$.
\end{prop}

\begin{proof} Write \[ \pi_1(\tilde\Sigma') = \la
x_1,y_1,x_2,y_2\ra,
\] then the moduli space $\cM(\tilde\Sigma',c)$,
$c\in \Cl$, consists of all conjugacy classes of $\rho$ satisfying
\[
\rho([x_1,y_1][x_2,y_2]) \in c.
\]
By Auerbach's Generation Theorem \cite[Theorem
6.82]{hofmann-morris2006_CompactGroups}, we have $G = \overline{\la
g_1',h_1' \ra}$ for some $g_1',h_1'$. Choose $g_1, h_1$ such that $g_1^m = g_1'$ and $h_1^m = h_1$. By Got{\^o}'s
Commutator Theorem \cite[Theorem
6.55]{hofmann-morris2006_CompactGroups}, we find $g_2,h_2\in G$ such
that
\[
[g_1,h_1][g_2,h_2]\in c.
\]
Consider the representation (see Figure \ref{FigSurface1}) determined on the generators by
\begin{align*}\tilde\rho\co \pi_1(\tilde\Sigma') & \to G\\
x_i & \mapsto g_i\\
y_i & \mapsto h_i.
\end{align*}

Clearly, $x_1^m, y_1^m \in \im(\pi\co \pi_1(\Sigma')\to \pi_1(\tilde\Sigma'))$ and therefore $g_1',h_1' \in \im(\rho)$. We claim that $\rho\coloneqq \pi^*\tilde\rho$ is irreducible, i.e. $\Stab_G(\rho) = Z(G)$. We
automatically have $Z(G) \subset \Stab_G(\rho)$. Let $g \in
\Stab_G(\rho)$. Then $g$ is in particular in the centralizer $C_G(\{
g_1' , h_1' \}) =C_G(\la g_1' ,h_1'\ra)$. Therefore by continuity
\[g\in C_G(\overline{\la g_1',h_1'\ra}) = C_G(G) = Z(G).\] Therefore $\rho$ is
irreducible.\kommentar{By Auerbach's Generation Theorem the set
$\{(g,h)\subset G \times G \mid G = \overline{\la g,h \ra} \}$ is
dense in $G\times G$, therefore $\cM'(\Sigma,c)$ is open dense in
$\cM(\Sigma,c)$.}
\end{proof}

A glance at Figure \ref{FigSurface1} gives the following.

\begin{cor} Let $\tilde\Sigma\coloneqq \Sigma/\la f\ra$ be a surface of genus greater than 1.
Then $\cM''(\tilde\Sigma',c(\delta))$ is nonempty for all $\delta\in\Delta$.
\end{cor}

As mentioned above, $\Delta \to C$ is surjective for $G=\SU(n)$, so that together with the above corollary we get the following.

\begin{thm}\label{ThmValidity}  If $\tilde \Sigma\coloneqq \Sigma/\la f\ra$ is a surface of genus greater than 1 and $G = \SU(N)$, then $\cM(\Sigma_f)_c' = r^{-1}(|\cM(\Sigma)|'_c)$ is nonempty for every $c\in C$.
\end{thm}

We will see in the next section, that $|\cM(\Sigma)|'_c$ being nonempty enables us to express the leading order term of the corresponding summand in the expression \eqref{WRTInvariant} as an integral over $|\cM(\Sigma)|'_c$. Theorem \ref{ThmValidity} therefore shows, in which cases we get an integral expression for the entire leading order term of the asymptotic expansion.

\section{The leading order term of the Witten-Reshetikhin-Turaev
invariant}

\label{LeadingTerms}

Let us now identify the leading
order term of the asymptotic expansion \eqref{WRTInvariant} of the
Witten-Reshetikhin-Turaev invariants of finite order mapping tori as
an integral of differential geometric terms.

We first consider the framing correction term as defined in \cite[Equation (5)]{andersen95}
\[
\det(f)^{\alpha} \coloneqq \tr (\tilde f \co \cL_{D,\sigma}^{\alpha}
\to\cL_{D,\sigma}^{\alpha}),
\]
where $\tilde f$ is a lift of $f$ to the rigged mapping class group determined by the Atiyah 2--framing. The rigged mapping class group is a central extension of the mapping class group constructed by \cite{walker1991:WittensInvariants} and \cite{turaev2010_QuantumInvariants} (see also \cite[Section 2]{andersen95}). This element $\tilde f$ acts on any power, say $\alpha$, of the determinant
line bundle $\cL_D$ over Teichm\"uller space and $\sigma$ is a point in Teichm\"uller space preserved by $f$. For the rest of this paper we denote $\Sigma$ with the complex structure $\sigma$ simply by $\Sigma$. The framing correction term is obtained by setting $\alpha= \zeta$.
\begin{prop}\label{prop1} For a finite order automorphism $f\co \Sigma \to
\Sigma$ of a surface $\Sigma$ we have
\[\det(f)^{\alpha} = \exp\left(\sum_{0\neq \tilde\omega \in
(-\frac{1}{2},\frac{1}{2})} -2\pi i \alpha \tilde\omega_i\right),\]
where $e^{2\pi i \tilde \omega_j}$, $\tilde\omega_j \in
[-\frac{1}{2},\frac{1}{2})$, are the eigenvalues of the pull-back
$f^*\co H^{1,0}(\Sigma,\bar\partial) \to
H^{1,0}(\Sigma,\bar\partial)$.
\end{prop}

\begin{proof} Let us identify $H^1(\Sigma,\R)$ with
$H^{1,0}(\Sigma,\bar\partial)$ via
\[
H^1(\Sigma,\R) \hookrightarrow H^1(\Sigma,\C)
\stackrel{\text{pr}}{\to} H^{1,0}(\Sigma,\bar\partial),
\]
where $\text{pr}$ is the projection to the subspace. We get that the diagram
\[
\begin{diagram}
\node{H^{1,0}(\Sigma,\R)}\arrow{s,l}{\cong} \node{H^{1,0}(\Sigma,\R)}\arrow{w,t}{f^*}\arrow{s,r}{\cong}\\
\node{H^1(\Sigma,\R)} \node{H^1(\Sigma,\R)}\arrow{w,t}{f^*}
\end{diagram}
\]
commutes. By naturality of Poincar\'e duality, the diagram
\[
\begin{diagram}
\node{H^1(\Sigma,\R)}\arrow{s,l}{\PD} \node{H^1(\Sigma,\R)}\arrow{w,t}{f^*}\arrow{s,r}{\PD}\\
\node{H_1(\Sigma,\R)} \arrow{e,t}{f_*} \node{H_1(\Sigma,\R)}
\end{diagram}
\]
commutes. In particular, the eigenvalues of $\PD^{-1} \circ f_*
\circ \PD$ and $f^*$ are inverses of each other. In analogy to
\cite[Section 5]{andersen95} we get that
\[
\det(f)^{\alpha} = \exp\left(\sum_{0\neq \omega \in
(-\frac{1}{2},\frac{1}{2})} - 2\pi i\alpha\tilde\omega_i\right)
\]
where $e^{-2\pi i \tilde\omega_j}$, $\tilde\omega_j \in
[-\frac{1}{2},\frac{1}{2})$, are the eigenvalues of $\PD^{-1}\circ
f_* \circ\PD$, or equivalently, where $e^{2\pi i \tilde\omega_j}$,
$\tilde\omega_j \in [-\frac{1}{2},\frac{1}{2})$, are the eigenvalues
of $f^*$.
\end{proof}

We now turn to the contribution from each component of the fixed point variety which contains irreducible connections.
Let $c\in C$ with $|\cM(\Sigma)|'_c$ nonempty and
consider $\omega_c^{d_c} \cap \tau_i(L_\bullet^c(\cO_{\cM(\Sigma)}))\in H_*(|\cM(\Sigma)|'_c).$ In order to give a formula for the top degree term of this element, we need to fix a complex structure on $\Sigma$ which is preserved by $f$.  This induces the structure of an algebraic projective variety on $\cM(\Sigma)$ and hence also on $|\cM(\Sigma)|_c$. Algebraic varieties have fundamental classes and we denote the fundamental class of $|\cM(\Sigma)|_c$ by $[|\cM(\Sigma)|_c]$. As described in Appendix \ref{LRR}, the Lefschetz-Riemann-Roch Theorem in \cite[Section 0.6]{baum-fulton-quart_LefschetzRiemann-Roch} gives
\[\tau_\bullet(L_\bullet^c(\cO_{\cM(\Sigma)})) = \Ch^\bullet(\lambda^c_{-1}{\cM(\Sigma)})^{-1}
\cap[|{\cM(\Sigma)}|_c]^{\Td}\in H_\bullet(|\cM(\Sigma)|'_c),\]
where $[|{\cM(\Sigma)}|_c]^{\Td}$ is the Todd fundamental class defined in \cite{baum-fulton-macpherson1979_Riemann-Roch} and $\lambda^c_{-1}{\cM(\Sigma)}$ is a certain element in the $K$-theory of $|\cM(\Sigma)|_c$ with complex coefficients also defined in \cite{baum-fulton-macpherson1979_Riemann-Roch} (see also \cite[Section 8]{andersen95} for a computation of this element in the case at hand). We recall that the highest degree term of $[|{\cM(\Sigma)}|_c]^{\Td}$ equals $[|{\cM(\Sigma)}|_c]$. The top degree is $
d_c = \dim_\C |\cM(\Sigma)|'_c$, so
the contribution from $\Ch(\lambda_{-1}^c{\cM(\Sigma)})^{-1}$ to the top degree term of $\omega_c^{d_c} \cap \tau_i(L_\bullet^c(\cO_{\cM(\Sigma)}))$ will simply be its degree zero part.
 Following \cite[Section
8]{andersen95}, we have $\lambda_{-1}^c {\cM(\Sigma)} = L^\bullet
\left(\sum (-1)^i[\Lambda^i \cN^*_c]\right)$, where $ \cN^*_c$ is the conormal sheaf to $|{\cM(\Sigma)}|_c$ (thought of as an $f$-equivariant sheaf) and $L^\bullet$ is the
homomorphism determined by $L^\bullet(E_a) = [E_a] \tensor a \in K^0(|{\cM(\Sigma)}|_c)
\tensor \C$ for an $a$--eigensheaf $E_a$ of $f$. Since $f$ is finite order, $\cN^*_c$ splits as the direct sum $\cN^*_c = \bigoplus_j
\cN^*_{c,j}$ of $a_j$--eigensheaves $\cN^*_{c,j}$ of $f$, where
$a_j = e^{2\pi i \frac{j}{m}}$ and $j=1,\ldots,m-1$, we then have
\[
L^\bullet(\cN^*_c) = \sum_{j=1}^{m-1} \cN^*_{c,j} \tensor a_j.
\]
Then the degree zero part of $\Ch(\lambda_{-1}^c{\cM(\Sigma)})^{-1}$ is
\[
\lambda_{-1}(\Rank\cN^*_c)^{-1} = \prod_{i=1}^{m-1} (1-a_i)^{-r_i} =
\frac{1}{\det(1-df|_{\cN^*_c})}, \quad r_i = \Rank\cN^*_{c,i}.
\]
This shows the following.
\begin{prop}\label{prop2}
\[\omega^{d_c} \cup \Ch^\bullet(\lambda^c_{-1}{\cM(\Sigma)})^{-1}
 = \frac{\omega^{d_c}}{\det(1-df|_{\cN^*_c})}.\]
\end{prop}
If $|\cM|'_c$ is not empty, then it is open and dense in $|\cM|_c$, so that we can integrate the above differential form over $|\cM|'_c$. Furthermore, any sensible integral over $|\cM|_c$ is equal to the integral over $|\cM|_c'$. Therefore the expression of the leading order term of the
Witten-Reshetikhin-Turaev invariants for each $c\in C$ in
differential geometric terms is an immediate consequence of
Proposition \ref{prop1} and \ref{prop2}.
\begin{thm}\label{ThmReformulationWRT} Let ${|\cM(\Sigma)|_c}$ be a connected component of $|\cM(\Sigma)|$ containing irreducible connections, then
\begin{equation*}\begin{split}
k^{-d_c}\det(f)^{\frac{1}{2}\zeta} &e^{2\pi i k \CS_{\Sigma_f}(c)}
\frac{1}{d_c!} (\exp(k\omega_c) \cap
\tau_\bullet(L_\bullet^c(\cO_{\cM(\Sigma)})))\\
&=\exp\left({i \pi \zeta\sum_{0\neq
\tilde\omega_j\in(-\frac{1}{2},\frac{1}{2})} \tilde\omega_j}\right)
e^{2\pi i k \CS_{\Sigma_f}(c)} \int_{a\in|\cM(\Sigma)|'_c}
\frac{1}{d_c!}\frac{(\omega_c)_{[a]}^{d_c}}{\det(1-df|_{\cN^*_{[a]}}
)} + O(\frac{1}{k}),
\end{split}
\end{equation*}
where $e^{2\pi i \tilde \omega_j}$, $\tilde\omega_j \in
[-\frac{1}{2},\frac{1}{2})$, are the eigenvalues of $f^*\co
H^{1,0}(\Sigma,\bar\partial) \to H^{1,0}(\Sigma,\bar\partial)$ and
$\cN^*_{[a]} \coloneqq \cN_{c,[a]}^*$ is the fiber over $[a] \in
|\cM(\Sigma)|_c$ of the conormal sheaf $\cN^*_c$ of
$|\cM(\Sigma)|_c$.
\end{thm}

Notice, that a connected component may contain more than one
irreducible component (in the Zariski topology). These components
can be of different dimensions, but only the components of dimension
$d_c$ will contribute to the integral.

\section{Reidemeister torsion}\label{ReidemeisterTorsion}

In this section we will summarize some basic facts about
Reidemeister torsion, which is a term in the asymptotic expansion of
the Witten-Reshetikin-Turaev invariants. To keep the proofs less
technical we will consider it as a density.  Note that, it is
possible and could be interesting to lead this discussion in the
context of sign-determined Reidemeister torsion as defined in
\cite{turaev2002} (see for example \cite{dubois2006}).

\subsection*{Torsion of a complex}

The notation has been adapted from \cite{freed92} and
\cite{jeffrey92}. Let $F$ be either $\R$ or $\C$. Let $L$ be a
1--dimensional vector space over $F$. We will denote by $L^{-1}$ the
dual of a complex line $L$ and by $l^{-1}\in L^{-1}$ the inverse of
$l$ given by $l^{-1}(l) = 1$. By a density on $L$ we mean a function
\[
|\cdot| \co L \to \R \quad \text{such that} \quad |c\omega| = |c|
|\omega| \text{ for } c\in F, \omega \in L.
\]
We denote the densities on $L$ by $|L^*|$. For an $n$--dimensional
vector space $V$ over $F$ we let $\det V = \Lambda^n V$ and define a
density on $V$ to be an element of $|\det V^*|$. A density on a
manifold $M$ is a section of the density bundle $|\det T^*M|$. Every
volume form $\omega$ on $V$ gives a density $|\omega|$. If we choose
an orientation, we can identify densities with volume forms.

\begin{defn}\label{DefnTorsion} Given a finite
cochain complex $(C^\bullet, d)$ of finite-dimensional complex
vector spaces, we denote \[ \det C^\bullet = \bigotimes_{j=0}^n
(\det C^j)^{(-1)^j}.
\]
Then the torsion
\[
\tau_{C^\bullet, d} \in \left|(\det C^\bullet)^{-1} \tensor (\det
H^\bullet(C,d))\right|
\]
is given by
\[
\tau_{C^\bullet,d} = \bigotimes_{j=0}^n \left(\left|d s^{j-1} \wedge
s^j \wedge \hat h^j\right|^{(-1)^{j+1}} \tensor
\left|h^j\right|^{(-1)^j}\right),
\]
after an arbitrary choice of \begin{itemize}
\item $s^j \in \bigwedge^{k_j} C^j$ with $d s^j \neq 0$, where $k_j$ is the rank of $d\co C^j \to C^{j+1}$,
\item $h^j \in \det H^j(C)$ non-zero and
\item a lift $\hat h^j \in \bigwedge^{l_j} C^j$ of $h^j$, where $l_j = \dim H^j(C^\bullet,d)$.
\end{itemize}
\end{defn}

We will use the {\em Multiplicativity Lemma} as our main
computational tool.

\begin{lem}\label{MultiplicativityLemma} Let
\begin{equation}\label{ses}
0\to C^\bullet_1 \stackrel{\nu^\bullet}{\to} C^\bullet_2
\stackrel{\mu^\bullet}{\to} C^\bullet_3 \to 0
\end{equation}
be a short exact sequence of cochain complexes, choose compatible
volume elements $\omega_i^\bullet$ in $C^\bullet_i$---that is,
$\omega^j_2 = \nu^*(\omega^j_1) \wedge \omega'^j$ with
$\mu^*(\omega'^j) = \omega^j_3$ for $\omega^j_i \in \det C^j_i$---,
and let $H^\bullet$ be the long exact sequence associated to
\eqref{ses}. Then
\[
\tau_{C_2^\bullet}(\omega_2) = \tau_{C_1^\bullet}(\omega_1)\cdot
\tau_{C_3^\bullet}(\omega_3)\cdot \tau_{H^\bullet},
\]
where $\omega_i = \prod (\omega_i^j)^{(-1)^j}$.
\end{lem}

For a proof see \cite[Corollary 1.20]{freed92} or \cite[Theorem
3.2]{milnor66}.

\subsection*{The Wang exact sequence}

In order to compute the Reidemeister torsion, we will employ the
Wang exact sequence \cite{wang1949:WangExactSequence,spanier1981:AlgTop}.

Let $(C^\bullet,d) = \bigoplus_{i=0}^{n}(C^i,d^i)$ be a chain
complex and $f^\bullet = \{f^i \co (C^i,d^i) \to (C^i,d^i)\}$ be a
chain map. Then the algebraic mapping torus
$(T^\bullet(f^\bullet),d_f)$ is the cochain complex with $T^i(f)
\coloneqq C^i \oplus C^{i-1}$ and boundary operator $d_f^i(x,y)
\coloneqq (d^{i}(x),-d^{i-1}(y) + \mu^{i}(x))$, where $\mu^\bullet =
\Id^\bullet - f^\bullet\co C^\bullet \to C^\bullet$. It is not
difficult to confirm, that we get a short exact sequence
\begin{equation}\label{WangShort} 0 \to (C^{\bullet-1},-d)
\stackrel{\nu^{\bullet-1}}{\to} (T^\bullet(f^\bullet),d_f)
\stackrel{\pi^\bullet}{\to} (C^{\bullet},d)\to 0
\end{equation}
of chain complexes, where $\nu^\bullet$ is the inclusion into first
summand and $\pi^\bullet$ is the projection onto the second summand.
Observe that $(C^{\bullet},-d)$ and $(C^\bullet,d)$ are isomorphic
chain complexes and that $H^i(C^{\bullet-1},d) =
H^{i-1}(C^\bullet,d)$. This yields a long exact sequence
$H^\bullet_W$ by the name Wang exact sequence
\[
\cdots \to H_W^{i}(C^\bullet) \stackrel{\mu^i}\to H_W^i(C^\bullet)
\stackrel{\nu^i}{\to} H_W^{i+1}(T^\bullet(f))
\stackrel{\pi^{i+1}}{\to} H_W^{i+1}(C^\bullet) \to \cdots.
\] It is easy to check, that the
boundary map is indeed induced by $\mu^\bullet$. Together with the
Multiplicativity Lemma \ref{MultiplicativityLemma} we get the
following useful result.
\begin{cor}\label{MappingTorusCor} Let $\omega^j \in \det C^j$ be a volume form for all $j$ and let $\omega \coloneqq \prod (\omega_i^j)^{(-1)^j}$.
Then we have  $\tau_{C^\bullet(M)}(\omega) =
\tau_{C^{\bullet-1}}(\omega^{-1})$ and therefore for $\omega_f =
\nu^*(\omega)\wedge \omega'$ with $\pi^*(\omega') = \omega^{-1}$
\[
\tau_{C^\bullet(M_f)}(\omega_f) = \tau_{H^\bullet_W}.
\]
In particular, this is independent of the choice of $\omega$.
\end{cor}

\subsection*{Reidemeister torsion}

If each $C^j$ comes equipped with a volume form, then the torsion is
an element of $|\det H^\bullet(C^\bullet,d)|.$ If $X$ is a smooth
manifold, $W$ an inner product space and $\rho \co \pi \to \GL(W)$ a
representation of $\pi=\pi_1(X)$, then we can consider the cellular
chain complex with local coefficients in $W$ twisted by $\rho$ given
by
\[C^\bullet(X,W_{\rho}) = \hom_{\Z\pi}(C_\bullet(\widetilde X),
W),\] where $\tilde X$ is the universal cover of $X$. Note that
$C_\bullet(\widetilde X)$ has a natural inner product, by which the
cells are orthonormal. If furthermore $\rho$ preserves the inner
product on $W$, then $C^\bullet(X,W_\rho)$ carries an induced inner
product and therefore volume forms. Then the {\em Reidemeister
torsion of $X$} is a density given by
\[
\tau_X(W_\rho) = \tau_{(C^\bullet(X,W_\rho),d)} \in  \left|\det
H^\bullet(C^\bullet,d)\right|
\]
and is independent of the choice of the cell decomposition of $X$.
The use of cochain complexes rather than chain complexes in defining
Reidemeister torsion simplifies the notation in our arguments
considerably when interpreting the torsion in terms of twisted de
Rham cohomology. Even though we need to choose a multiple of the
Killing form as a metric on $\fg$ in order to identify Reidemeister
torsion defined through chains and Reidemeister torsion defined
through cochains, it is not difficult to see that the identification
is independent of this choice. If $A$ is a $G$-connection and the
representation $\ad\circ\hol(A)=\rho$ is associated to a flat
$G$--connection $A$ via the adjoint representation
\[
\ad\co G \to \O(\fg^\C) \subset \End(\fg^\C),
\] which takes values in the orthogonal group with respect to the Killing form on $\fg$, we
define
\[\tau_X(A)\coloneqq \tau_X(\fg_{\rho}).\]
Note that we can also consider the complexified adjoint
representation
\[
\Ad\co G \to \U(\fg^\C) \subset \End(\fg^\C),
\]
where we have extended the Killing form to a sesquilinear form on
$\fg^\C$. We then also have \[\tau_X(A) =
\tau_X(\fg^\C_{\Ad\circ\hol(A)}).\]

\section{Reidemeister torsion of mapping
tori}\label{ReidemeisterTorsionMappingTori}

We will see in this section that for $c \in C$
\begin{equation}\label{ReidemeisterIntegral}
\int_{\cM(\Sigma_f)'_{c}} \tau_{\Sigma_f}(A)^{\frac{1}{2}} =
|Z(G)|\int_{|\cM(\Sigma)|'_c}
\frac{|\omega_c^{d_c}|}{|\det(1-df|_{\cN^*_{[a]}})|}, \end{equation}%
where $\cN^*_{[a]} = \cN^*_{c,[a]}$ and the conormal sheaf $\cN^*_c$
is the dual of the normal sheaf
\[
\cN_c = \frac{T{\cM(\Sigma)} |_{|\cM(\Sigma)|_c}}{T
{|\cM(\Sigma)|_c}}.
\]
In the above equation we identified $H^2(\Sigma_f,d_A)$ with
$H^1(\Sigma_f,d_A)^*$ via $\PD$. In fact, we will even show an
equality for irreducible components on the level of densities. The
factor $|Z(G)|$ then stems from the fact that $r\co \cM(\Sigma_f)
\to |{\cM(\Sigma)}|$ is a $|Z(G)|$--sheeted covering map (see
\cite[Section 7]{andersen95}).

Notice, that $T|{\cM(\Sigma)}|$ is simply the kernel of the bundle
map
\[1-d f^* \co T {\cM(\Sigma)} \to T {\cM(\Sigma)}\] and is therefore isomorphic to the bundle of
1--eigenspaces of
\[f^*\co T_{[a]}{\cM(\Sigma)} \to T_{[a]}{\cM(\Sigma)}, \text{
where } [a] \in |{\cM(\Sigma)}|_c.\] We can fix an isomorphism
$H^{0,1}(\Sigma,\bar\partial_a)\cong T_{[a]}{\cM(\Sigma)}$ to get an
equivalent statement for $H^{0,1}(\Sigma,\bar\partial_a)$. Also,
note that the eigenvalues of $1-df^*\co \cN_{[a]} \to \cN_{[a]}$ and
of $1-df\co \cN^*_{[a]} \to \cN^*_{[a]}$ are the same, where $df$ is
short for $(df^*)^*$.

\subsection*{General mapping tori}

Consider a CW complex $M$ and an orientation preserving simplicial
homeomorphism $f\co M \to M$. The torsion for the mapping torus
$M_f$ of $f$ has been computed in \cite[Proposition 3]{fried83} (see
also \cite[Section 6.2]{felshtyn2000} and \cite[Example
2.17]{nicolaescu2003}) only when $M_f$ is an acyclic CW complex. In
this section we will give a generalization of the computation for
mapping tori to the non-acyclic case. The computations in
\cite{dubois2006} of sign-determined Reidemeister torsion for
fibered knots for the local coefficient systems $\fsu(2)$ and
$\fsl_2(\C)$ use the same basic tools, namely the Wang exact
sequence and the Multiplicity Lemma.

Let $\rho \co \pi_1 M_f \to G$ be a $G$--representation of $\pi_1
M_f$ acting on $\fg$ by the adjoint representation. If we denote by
$C_g \co G \to G$ the conjugation action, then $\rho$ is determined
by a representation $\rho'\co \pi_M \to G$ satisfying $\rho' = C_g
\circ (f^* \rho')$ for some $g\in G$. The choice of $g$ induces a
chain map $f^\bullet =f^\bullet_g\co C^\bullet(M,\fg_{\rho'}) \to
C^\bullet(M,\fg_{\rho'})$. It is easy to check that the algebraic
mapping torus $T(f^\bullet)$ is isomorphic---in fact, isometric---to
$C^\bullet(M_f,\fg_\rho)$ induced by the cell decomposition of $S^1$
into two cells and $C^\bullet(M)$.

In this section let us from now on drop the coefficients in the
cohomology and cochain groups entirely with the understanding that
we consider coefficients twisted by representations compatible with
the restriction. Instead of $\mu^i$ and $\pi^i$ we will sometimes
use the more familiar notation $\mu^*$ and $\pi^*$, when the grading
is clear. Consider the diagram in cohomology induced by the Wang
exact sequence and a positive multiple $\Theta$ of Poincar\'e
duality on $M$
\begin{equation*}
\dgARROWLENGTH=1em
\begin{diagram}
\node{\hspace{1.5cm}\cdots} \arrow{e}\node{H^{n-i}(M_f)}
\arrow{e,t}{\pi^*}\arrow{s,l,..}{\Theta}
\node{H^{n-i}(M)}\arrow{e,t}{\mu^*}\arrow{s,l}{\Theta}\node{H^{n-i}(M)}\arrow{e,t}{\nu^*}\arrow{s,l}{\Theta}\node{H^{n-i+1}(M_f)}\arrow{e}\arrow{s,l,..}{\Theta}\node{\cdots\hspace{1.5cm}}\\
\node{\hspace{1.5cm}\cdots}
\arrow{e}\node{\left(H^{i+1}(M_{f})\right)^*} \arrow{e,t}{\nu}
\node{\left(H^{i}(M)\right)^*}
\arrow{e,t}{\mu'}\node{\left(H^{i}(M)\right)^*} \arrow{e,t}{\pi}
\node{\left(H^{i}(M_{f})\right)^*}\arrow{e}\node{\cdots,\hspace{1.5cm}}
\end{diagram}
\end{equation*}
where $\mu' = 1-f^{-1}$ and we write $f^{-1} =((f^{-1})^*)^*$. It is
easy to check that the middle square commutes. Furthermore, since
\[\mu' = (1-f)\circ (-f^{-1}) = \mu \circ (-f^{-1}) =
(-f^{-1})\circ \mu,\] and $(-f^{-1})$ is an isomorphism, the
exactness of the above sequence implies the exactness of the lower
sequence. We can define isomorphisms $\Theta\co \im \nu^* \to \im
\pi$ so that the above diagram commutes, and we can extend these
maps arbitrarily to isomorphisms $\Theta\co H^{n-i}(M_f) \to
H^{i+1}(M_f)$. We extend $\Theta$ to the exterior algebra by setting
$\Theta(a\wedge b) = \Theta(a)\wedge \Theta(b)$.

Before we can compute Reidemeister torsion of a general mapping
torus, we need a few technical facts. For finite order mapping tori
the situation simplifies considerably and the result is more
pleasing.

\begin{lem}\label{technical} Let $0 \neq h^{i+1} \in \det(\im(\pi^i))$. Then we can find
$h^i_+\wedge h^i_- \in \det(H^i(M))$ such that $\nu^*(h^i_-) =
h^{i+1}$ and $\mu^*(h^i_+) \wedge h^i_- \neq 0$.
\end{lem}

\begin{proof} Let $h^i_+\wedge h^i_- \in \det(H^i(M))$ such that $\nu^*(h^i_-) = h^{i+1}$. If $\mu^*(h^i_+) \wedge
h^i_- =0$, then let $k^i \in \Lambda(H^i(M))$ with $0\neq \mu^*(k^i)
\wedge h^i_-\in \det (H^i(M))$. Now choose $\lambda
> 0$ small enough that for $\tilde h^i_+ \coloneqq h^i_+ +
\lambda k^i$ \[\tilde h^i_+ \wedge h^i_- \neq 0.\] Then we also have
\[\mu(\tilde h^i_+) \wedge h^i_- =\lambda \mu(k^i) \wedge
h^i_- \neq 0.\qedhere\]
\end{proof}

\begin{prop}\label{torsionGeneral}
Let $M_f$ be a mapping torus of a homeomorphism $f\co M \to M$,
$\dim M = n$. Then we may choose $h^i \in \Lambda(H^i(M_f))$ and
$h^i_-, h^i_+ \in \Lambda(H^i(M))$ for all $i$ with
\begin{equation}\label{hypothesis}
\begin{split}
0\neq \nu^*(h^{i-1}_-) \wedge h^i &\in
\det(H^i(M_f)),\\0\neq \pi^*(h^i) \wedge h^i_+ &\in \det(H^i(M)),\\
\text{and}\quad  0\neq \mu^*(h^i_+) \wedge h^i_- &\in
\det(H^i(M)).\end{split}\end{equation} so that they satisfy
\begin{equation}\label{conditions}
|\Theta(\nu^*(h^{n-i}_-))(h^i)| = 1 \quad \text{and} \quad |h_-^i
\wedge h_+^i| = |\pi^*(h^{i})\wedge h^i_+|,
\end{equation}
Furthermore, the Reidemeister torsion is
\[\tau(M_f) = \left|\bigotimes_{i=0}^{n+1} (\nu^*(h^{i-1}_-) \wedge h^i)^{(-1)^{i}}\right| \prod_{i=0}^{n} |\det(\tilde\mu^i)|^{(-1)^{i+1}}.\]
where $\tilde\mu^i$ is determined by \[\tilde \mu^i (h^i_- \wedge h^i_+) =h^i_- \wedge \mu^*(h^i_+).\]
\end{prop}

\begin{proof}
The Wang exact sequence and Lemma \ref{technical} allow us to choose
$h^i \in \Lambda(H^i(M_f))$ and $h^i_-, h^i_+ \in \Lambda(H^i(M))$
for all $i$ with
\begin{align*}
0\neq &\nu^*(h^{i-1}_-) \wedge h^i \in
\det(H^i(M_f)),\\
0\neq &\pi^*(h^i) \wedge h^i_+ \in \det(H^i(M)),\\
0\neq &\mu^*(h^i_+) \wedge h^i_- \in \det(H^i(M)),\\
\tag*{\text{and}}0 \neq &h_+^i\wedge h_-^i .\end{align*} By
rescaling we can assume $|\Theta(\nu^*(h^{n-i}_-))(h^i)|= 1$. Notice
that, if $h^i$ and $h^{n-i}_-$ satisfy this condition, so do
$\lambda h^i$ and $\frac{1}{\lambda}h^{n-i}_-$ for $\lambda >0$. By
choosing $\lambda$ appropriately we may therefore assume that
\[|h_-^i \wedge h_+^i| = |\pi^*(h^{i})\wedge h^i_+|.\] Then
\[|\det \tilde \mu^i| \cdot |\pi^*(h^{i})\wedge h^i_+| = |\det \tilde \mu^i |\cdot |h_-^i \wedge h_+^i| = |\tilde \mu^i (h^i_- \wedge h^i_+)| =|h^i_-
\wedge \mu^*(h^i_+)|,\] and therefore \[\bigotimes_{i=0}^n
|\pi^*(h^{i})\wedge h^i_+|^{(-1)^{i}} \bigotimes_{i=0}^n
|\mu^*(h^i_+) \wedge h^i_-|^{(-1)^{i+1}} = |\det \tilde
\mu^i|^{i+1}.
\]
By Corollary \ref{MappingTorusCor}, the proposition follows.
\end{proof}

Note, that even though the system of equations \eqref{conditions}
seems to be overdetermined, half of them are equivalent to the other
half, since the above diagram is commutative. Also observe that,
even though our result seems to depend on $\Theta\co H^{n-i}(M)\to
H^{i+1}(M)$, we can use a different multiple of Poincar\'e duality
without changing Reidemeister torsion: This can be easily verified
by the skeptical reader by considering the cases $n$ odd and $n$
even separately.

\subsection*{Finite order mapping tori}

We can enhance Theorem \ref{torsionGeneral} and make it more useful for finite order mapping tori, if we put
some restrictions on $\mu^*$. Before we do that, let us state and prove a simple fact from linear algebra.

\begin{lem}\label{LinAlgLemma}
For a linear map $T\co V \to V$ between finite-dimensional vector spaces, the following are equivalent
\begin{enumerate}\item $\bar T \co V/\ker T \to V/\ker T$ induced by $T$ is an isomorphism. \item $\hat T =
T|_{\im T} \co \im T \to \im T$ is an isomorphism.
\end{enumerate}
Furthermore $\det \bar T = \det \hat T$.
\end{lem}

\begin{proof}
Clearly, $\bar T$ is an isomorphism if and only if $\im T \hookrightarrow V \to V/\ker T$ is an isomorphism. The
last statement is equivalent to $\im T \cap \ker T = 0$. This implies that $\hat T\co \im T \to \im T$ is an
isomorphism. On the other hand, if $0 \neq v \in \im T \cap \ker T$, then $\hat T$ is not injective, because
$T(v) = 0$.

Furthermore, if $\{b_i\}_i$ is a basis of $\im T$, then $\{[b_i]\}_i$ is a basis of $V/\ker T$. It follows
immediately, that $\det \bar T = \det \hat T$.
\end{proof}

\begin{prop}\label{torsionIso}
Assume that
\[\bar\mu^i\co H^i(M)/\ker(\mu^i)\to
H^i(M)/\ker(\mu^i)\] is an isomorphism. Then we can choose $h^i$
with $0\neq \pi^*(h^i) \in \det(\im\pi^i)$ satisfying
\begin{equation}\label{conditions2}
\Theta(\nu^*(\pi^*(h^{n-i})))(h^i) = 1.
\end{equation}
Furthermore, we have $\det(\bar \mu^i) = \det(\tilde \mu^i)$, where
$\tilde \mu^i$ is the map from Theorem \ref{torsionGeneral}. In
particular, if $\bar\mu^i$ is an isomorphism for all $i$---for
example for finite order mapping tori---we have
\[\tau(M_f) = \left|\bigotimes_{i=0}^{n+1} (\nu^*(\pi^*(h^{i-1})) \wedge
h^i)^{(-1)^{i}}\right| \prod_{i=0}^{n}
|\det(\bar\mu^i)|^{(-1)^{i+1}}.\]
\end{prop}

\begin{proof} Suppose that $\bar\mu^i$ is an isomorphism. Choose $h^i$ with $0\neq \pi^*(h^i) \in \det(\im \pi^i)$. In view of Lemma \ref{LinAlgLemma} we can find $h^i_+ \in \det \im
\mu^i$ with $0\neq \pi^*(h^i)\wedge h^i_+ \in \det H^i(M)$. Since
$h^i_+\in  \det \ker \nu^i$, we deduce $0 \neq
\nu^*\circ\pi^*(h^i)\in \det \im \nu^i$, which allows us to rescale
$h^i$ so that it satisfies \eqref{conditions2} above. If we set
$h^i_-\coloneqq \pi^*(h^i)$, it is straightforward to see that
\eqref{conditions} is satisfied and that \[ \det \tilde \mu^i = \det
\hat \mu^i.\qedhere
\]
\end{proof}

\subsection*{Finite order mapping tori of surfaces}

We will now focus on the case of a mapping torus $\Sigma_f$ of
finite order for a closed surface $\Sigma$. The goal of this section
is to identify integral of the square root of Reidemeister torsion
with the leading order term in formula \eqref{WRTInvariant} as
predicted by the semiclassical approximation of the path integral.
We will only do this for the case, when $c\in C$
contains an open, dense submanifold $|\cM(\Sigma)|'_c$ of
irreducible connections of $|{\cM(\Sigma)}|$. More specifically, we
will establish an identification on the level of densities for
$|\cM(\Sigma)|'_c$. Notice, that while the square root of
Reidemeister torsion is a density on a submanifold of the
irreducible connections of $\cM(\Sigma_f)$, $\omega^{d_c}$ is a
density on top-dimensional component of $|\cM(\Sigma)|'_c$.
Therefore the density on $|{\cM(\Sigma)}|$ needs to be pulled back
to a density on ${\cM(\Sigma)}$ via the natural restriction and
$|Z(G)|$--sheeted covering map $r\co \cM(\Sigma_f) \to
|{\cM(\Sigma)}|$  before we can relate it to Reidemeister torsion.
We also need to point out, that by treating Reidemeister torsion as
a density, we chose to identify $H^2(\Sigma_f,d_A)$ with
$(H^1(\Sigma_f,d_A))^*$ for $A \in \cA_{\Sigma_f}$ via $\PD$.

Before we prove the main theorem, we would like to mention the
following simple fact.
\begin{lem}\label{omegaProperty} Let $(V^{2n},\omega)$ be a symplectic vector space. We can identify $V$ with $V^*$ by \[\Theta(v)(w)\coloneqq -\omega(v,w)\] and extend this map to the exterior algebra by $\Theta(v\wedge w) \coloneqq \Theta(v)\wedge \Theta(w)$. Then the volume form $\vol = \tfrac{1}{n!} \omega^{n} \in \det
V^*$ on $V$ satisfies
\[\Theta(\vol^{-1})(\vol^{-1}) = 1,\]
where $\vol^{-1}\in \det V$ is given by $\vol (\vol^{-1}) = 1$.
\end{lem}

\begin{proof} Form a symplectic basis
$\{a_i,b_i\}_{i=1,\ldots, n}$ of $V$, that is, $\omega(a_i,b_j) =
-\omega(b_j,a_i) = \delta_{ij}$. Then we have
\[\omega = - \sum_{i=1}^{n} \Theta(b_i)\wedge
\Theta(a_i)=\sum_{i=1}^{n} \Theta(a_i)\wedge \Theta(b_i).
\]
Then
\[
\vol = \frac{1}{n!}\omega^{n} = \frac{1}{n!} \bigwedge^{n }
\sum_{i=1}^{n} \Theta(a_i)\wedge \Theta(b_i) =
\bigwedge_{i=1}^{n}\Theta(a_i)\wedge \Theta(b_i)
\]
as well as
\[vol^{-1} = (-1)^{n}\bigwedge_{i=1}^{n} b_i\wedge a_i =
\bigwedge_{i=1}^{n} a_i\wedge b_i.\] Therefore we get the desired
equation \[ \Theta(\vol^{-1})(\vol^{-1})  = \left(
\bigwedge_{i=1}^{n} (\Theta(a_i) \wedge \Theta(b_i))\right)
\vol^{-1} = \vol (\vol^{-1}) = 1. \qedhere\]
\end{proof}

\begin{thm} Let $A$ be an irreducible flat connection on $\Sigma_f$ such that $a\coloneqq r(A)$ is irreducible on
$\Sigma$ and $c\subset |{\cM(\Sigma)}|$ is a connected component
containing $a$. Let $\omega$ be the usual symplectic form on
$H^1(\Sigma,d_a)$ given by $\eqref{SymplecticForm}$ and identify
$H^2(\Sigma_f,d_{A})$ with $H^1(\Sigma_f,d_{A})^*$ via $\PD$. Then
we have
\[\tau_{\Sigma_f}(A)^{\frac{1}{2}} =
\frac{1}{d_c!}\frac{|r^*(\omega_c)^{d_c}|}{\sqrt{|\det(1-f^1)|}},\]
where $d_c = \frac{1}{2}(\dim_\R H^1(\Sigma_f,d_A)-\dim_\R
H^0(\Sigma_f,d_A))$ and the restriction $\omega_c$ of $\omega$ to
$\ker (1-f^1)$ is a symplectic form on $\ker (1-f^1)$.
\end{thm}

\begin{proof}
The Reidemeister torsion
\[\tau_{\Sigma_f}(A) \in \det
H^0(\Sigma_f,d_{A}) \tensor \det H^1(\Sigma_f,d_{A})^* \tensor \det
H^2(\Sigma_f,d_{A}) \tensor \det H^3(\Sigma_f,d_{A})^*\] has been
computed in Theorem \ref{torsionIso}. Since we are only
interested in $A$ irreducible, we have $H^0(\Sigma_f,d_{A}) =
H^3(\Sigma_f,d_{A}) = 0$. In contrast to \cite[Proposition 5.6]{jeffrey92}, where $f$ has
isolated fixed points on ${\cM(\Sigma)}$, we have to consider connected components
$c\subset|{\cM(\Sigma)}|$, which are positive-dimensional. Furthermore, $\PD$
identifies $H^2(\Sigma_f,d_{A})$ with the dual of
$H^1(\Sigma_f,d_{A})$ and $\dim H^1(\Sigma_f,d_{A}) = \dim
\cM(\Sigma_f)_c = \dim|{\cM(\Sigma)}|_c$. In summary we get
\[0\neq \sqrt{\tau_{\Sigma_f}(A)} \in |\det H^1(\Sigma_f,d_{A})^*|.\]

Since we also assume irreducibility of $a=r(A)$, $\ker(\pi^1) =
\im(\nu^0) = 0$. Furthermore, \[0 \neq \omega_c^{d_c} \in
\det(E_1(f^1))^* = \det(\ker(1-f^1))^* = \det(\ker(\mu^1))^* =
\det(\im(\pi^1))^*.\] Since $\pi^1$ is injective, we can define an
element $h^1 \in H^1(\Sigma)$ by requiring
\[\pi^*(h^1)=(\omega_c^{d_c})^{-1} \in \det(\im(\pi^1)).\] All that
is left to complete the proof of the theorem is that this choice of
$h^1$ indeed satisfies condition \eqref{conditions2}. Since
$H^2(\Sigma,d_a)=0$ we have $\det(\im
(\nu^1))=\det(H^2(\Sigma_f,d_A))$. Since the map $\bar \mu^1$ from
Theorem \ref{torsionIso} is an isomorphism and $0\neq \pi^*(h^1) \in
\ker
 \mu^1$, we have
\[0\neq \nu^*(\pi^*(h^1))\in \det(H^2(\Sigma_f,d_A)).\]

We would like to apply Theorem \ref{torsionIso}. Since we chose
$\PD$ to identify $H^2(\Sigma_f,d_{A}) = (H^1(\Sigma_f,d_{A}))^*$ we
need to check that $\Theta = \PD$ indeed satisfies condition
\eqref{conditions2}. We see, that condition \eqref{conditions2} is
equivalent to
\[
\PD((\omega_c^{d_c})^{-1})((\omega_c^{d_c})^{-1}) =
\Theta(\pi^*(h^{1}))(\pi^*(h^1))= \pi(\Theta(\pi^*(h^{1}))(h^1) =
\Theta(\nu^*(\pi^*(h^{1})))(h^1) = 1,
\]
which is satisfied by Lemma \ref{omegaProperty}.
\end{proof}

Geometrically, $T_{[a]}{\cM(\Sigma)} \cong
H^{0,1}(\Sigma,\bar\partial_a)$, and therefore $|\det(1-f^1)| =
|\det(1-df_{\cN^*_{r(A)}})|^2$, where we again understand $df$ as
$(d(f^*))^*$.

\begin{thm} \label{ReidemeisterTorsionComputation} Let $A$ be an irreducible flat connection on $\Sigma_f$ such that $r(A)$ is irreducible on
$\Sigma$. If we identify densities with volume forms using the
orientation induced by $r^*(\omega_c)_{A}^{d_c}$, we have
\[\sqrt{\tau_{\Sigma_f}(A)} =
\frac{1}{d_c!}\frac{r^*(\omega_c)_{A}^{d_c}}{|\det(1-df|_{\cN^*_{r(A)}}
)|}.\]
\end{thm}

This shows that over the moduli space of irreducible flat
connections $A$ with $r(A)$ irreducible we indeed have the identity
\eqref{ReidemeisterIntegral}.

\section{The \texorpdfstring{$\rho$}{rho}--invariant}\label{RhoInvariant}

Another classical topological invariant, which appears in the
expansion of the Witten-Reshetikin-Turaev invariants, is the
$\rho$--invariant. We will briefly review the definition for
$3$--manifolds in the context of the adjoint representation and
relate it to the original definition using the defining
representation before we state the result from \cite{bohn2009},
which will be relevant for us.

\subsection*{The Definition}

For a formally self-adjoint, elliptic differential operator $D$ of
first order, acting on sections of a vector bundle over a closed
manifold $X$, one defines the $\eta$--function
\begin{equation}\label{EqnEtaInvariant}
 \eta(D,s) \coloneqq \sum_{0 \neq \lambda \in \Spec(D)} \frac{\sgn(\lambda)}{|\lambda|^s}, \quad \re(s) \text{ large}.
\end{equation}
The function $\eta(D,s)$ admits a meromorphic continuation to the
whole $s$-plane with no pole at the origin. Then $\eta(D) \coloneqq
\eta(D,0)$ is called the $\eta$--invariant of $D$.

As a special case, let $G$ be a compact, simple, simply-connected
Lie group and $A$ a $G$--connection on a Riemannian 3--manifold $X$.
Then the odd signature operator coupled to $A$ is the formally
self-adjoint, elliptic, first order differential operator
\begin{equation}\label{OddSignature}
\begin{split}
D_A\co \Omega^{0}(X;\fg) \oplus \Omega^{1}(X;\fg)&\longrightarrow \Omega^{0}(X;\fg) \oplus \Omega^{1}(X;\fg)\\
(\alpha,\beta) &\longmapsto (d_A^* \beta, d_A \alpha+ *d_A \beta),
\end{split}
\end{equation}
where $d_A\co \Omega^p(X;\fg) \to \Omega^{p+1}(X;\fg)$ is the
covariant derivative associated to $A$ and $G$ acts on $\fg$ via the
adjoint action. If $A$ is flat, the $\rho$--invariant is given by
\begin{equation}\label{EqnRhoInvariant}\rho_A(X) \coloneqq \eta(D_A) - \eta(D_\theta),\end{equation} where
$\theta$ is the trivial connection. The $\rho$--invariant is
metric-independent and gauge-invariant. We write $\rho_{\hol(A)} =
\rho_A$, where the representation $\hol(A) \co \pi_1 X \to G$ is the
holonomy of $A$.

In the original definition \cite{atiyah-patodi-singer75b} by Atiyah,
Patodi and Singer, their $\rho$--invariant has been similarly
defined for a $\U(n)$--representation, where $\U(n)$ acts on $\C^n$
by the defining representation. We will briefly describe its
relationship to our definition of the $\rho$--invariant in
\eqref{EqnRhoInvariant}. With respect to an ad-invariant metric on
$\fg$---for example the Killing form---on $\fg$, the adjoint
representation takes values in the orthogonal endomorphisms of $\fg$
\[\ad\co G \to \SO(\fg) \subset \End(\fg).\] We can consider the complexified adjoint
representation
\[
\Ad\co G \to \SU(\fg^\C) \subset \End(\fg^\C).
\]
Then $\rho_{\hol(A)}$ is equal to the Atiyah-Patodi-Singer $\rho$--
invariant of $\Ad\circ\hol(A)$.

\subsection*{The Rho--invariant of finite order mapping tori}

Let $\Sigma$ be a surface and $P$ a principal $G$--bundle. In order
to make use of the results in \cite{bohn2009}, we consider the
bundle $\Ad P$ associated to the complexified adjoint
representation, which is a Hermitian vector bundle of rank $\dim G$.

The chirality operator $\tau_\Sigma$ on $\Omega^p(\Sigma)$ is given
by
\[
\tau_\Sigma = (-1)^{\frac{p(p-1)}{2}+2p}i*_p,
\]
where $*_p$ is the Hodge star operator on $\Omega^p(\Sigma)$. Note
that the splitting into $\pm 1$--eigenspaces of $\tau_\Sigma$
restricted to the harmonic forms $\cH^\bullet_a(\Sigma;\Ad P) = \ker
\Delta_a$ of $\Delta_a \coloneqq d_a d_a^* + d_a^* d_a$
\[
\cH^\bullet_a(\Sigma;\Ad P) = \cH^+_a(\Sigma;\Ad P) \oplus
\cH^-_a(\Sigma;\Ad P)
\]
is invariant under $\Phi f^*$ for any gauge transformation $\Phi \co
P \to P$ satisfying $\Phi f^* a = a$. Since the unitary structure on
$\Ad P$ arises from $\ad \co G \to \O(\fg)$ (see \cite[Remark (ii)
on page 136]{bohn2009}), we have
\begin{align*} \tr \log[\Phi f^*|_{\cH^+_a(\Sigma;\Ad P) \cap
\Omega^1}] &= \rk [(\Phi f^* - \Id)|_{\cH^-_a(\Sigma;\Ad P) \cap
\Omega^1}] - \tr \log[\Phi
f^*|_{\cH^-_a(\Sigma;\Ad P) \cap \Omega^1}],\\
\rk [(\Phi f^* - \Id)|_{\cH^+_a(\Sigma;\Ad P)\cap \Omega^1} &= \rk
[(\Phi f^* - \Id)|_{\cH^-_a(\Sigma;\Ad P)\cap\Omega^1},\\
\tag*{\text{and}}\rk[(f^*-\Id)|_{\cH^+(\Sigma) \cap \Omega^1}] &=
\rk[(f^*-\Id)|_{\cH^-(\Sigma) \cap \Omega^1}]\end{align*} where $\tr
\log$ is defined for a diagonalizable map $T$ as
\[
\tr\log T \coloneqq \sum_{j=1}^n\theta_j \in \R,
\]
where $e^{2\pi i\theta_j}$ are the eigenvalues of $T$, and where we
require $\theta_j \in [0,1)$. Also note that \[\cH^+_a(\Sigma;\Ad
P)\cap \Omega^1 = H^{1,0}(\Sigma,\bar\partial_a) \qquad \text{and}
\qquad \cH^-_a(\Sigma;\Ad P) \cap \Omega^1=
H^{0,1}(\Sigma,\bar\partial_a).\] A fixed isomorphism
$H^{0,1}(\Sigma,\bar\partial_a)\cong T_{[a]}\cM(\Sigma)$ gives the
commutative diagram \[
\begin{diagram}
\node{H^{0,1}(\Sigma,\bar\partial_a)}\arrow{s,l}{\cong} \arrow{e,t}{\Phi f^*}\node{H^{0,1}(\Sigma,\bar\partial_a)}\arrow{s,r}{\cong}\\
\node{T_{[a]}\cM(\Sigma)}
\arrow{e,t}{df^*}\node{T_{[a]}\cM(\Sigma),}
\end{diagram}
\]
and we have $\rk [(df^* - \Id)|_{T_{[a]}\cM(\Sigma)}] = \rk
\cN_{[a]}$. We simplify \cite[Theorem 4.2.4]{bohn2009} as follows.

\begin{thm} Let $f\co \Sigma \to \Sigma$ be
a finite order homeomorphism. Let $A$ be a flat $G$--connection over
$\Sigma_f$ with $r([A])=[a]$. Then
\begin{equation}\label{FiniteOrderRho}
\begin{split}
\rho_A(\Sigma_f) = -4 & \tr \log[df^*|_{T_{[a]}\cM(\Sigma)}] + 2\rk \cN_{[a]}\\
 & {}- 4 \dim G \tr
\log[f^*|_{\cH^{1,0}(\Sigma,\bar\partial)}] + 2 \dim G \rk
[(f^*-\Id)|_{\cH^{1,0}(\Sigma,\bar\partial)}].
\end{split}
\end{equation}
\end{thm}

\begin{rem}\label{RemarkEta} It follows from the proof of \cite[Theorem 4.2.4]{bohn2009}, that
\begin{align*}
\eta(D_A) &= -4  \tr \log[df^*|_{T_{[a]}\cM(\Sigma)}] + 2\rk
\cN_{[a]},\\
\tag*{\text{ and}}\eta(D_\theta) &= -4 \dim G \tr
\log[f^*|_{\cH^{1,0}(\Sigma,\bar\partial)}] + 2 \dim G \rk
[(f^*-\Id)|_{\cH^{1,0}(\Sigma,\bar\partial)}].
\end{align*}
\end{rem}

\section{Identifying the classical invariants}\label{identification}

In this section we identify the classical invariants in the leading
order term of the Witten-Reshetikhin-Turaev invariants
\eqref{WRTInvariant} of a finite order mapping torus $X=\Sigma_f$ as
conjectured by the stationary phase approximation
\eqref{stationaryphaseRho}. More precisely, since the leading order
term of
\[\zeta = \frac{k\dim G}{k+h} = \dim G - \frac{h\dim G}{k+h}\] is
simply $\dim G$, we identify the classical invariants in
the leading order term
\begin{equation}\label{WRTInvariantLead}
\det(f)^{-\frac{1}{2}\dim G} e^{2\pi i k \CS_{\Sigma_f}(c)}
\frac{1}{d_c!} (\omega_c^{d_c} \cap
\tau_{d_c}(L_\bullet^c(\cO_{\cM(\Sigma)})))k^{d_c}
\end{equation}
of \eqref{WRTInvariant} corresponding to an irreducible component
$|\cM(\Sigma)|_c$ of the variety $|\cM(\Sigma)|$ containing an
irreducible connection. Theorem \ref{ThmReformulationWRT} gives an
expression of \eqref{WRTInvariantLead} in terms of an integral over
$|\cM(\Sigma)|'_c$. We reformulate this as an integration of
classical invariants of 3--manifolds over $\cM(\Sigma)'_c$.

By Theorem \ref{ReidemeisterTorsionComputation} we have for $A \in
\cM(\Sigma_f)_c$
\begin{equation*}\sqrt{\tau_{\Sigma_f}(A)} =
\frac{1}{d_c!}\frac{r^*(\omega_c)_{A}^{d_c}}{|\det(1-df|_{\cN^*_{r(A)}})|},\end{equation*}
keeping in mind, that we have identified densities with volume forms
in the orientation induced by $r^*(\omega_c)^{d_c}_A$. Notice, that
for a complex root of unity $\xi = e^{2\pi i \theta}$ with $\theta
\in (0,1)$, we have $1-\xi = \xi(\xi^{-1}-1) = \xi(\bar \xi - 1)$.
Therefore
\begin{equation*}\left(\frac{1-\xi}{|1-\xi|}\right)^2 =
\frac{1-\xi}{1-\bar \xi} = -\xi.\end{equation*} By observing that
the real part of $1-\xi$ is always positive, we see that
\begin{equation}\label{SimpleFact}\frac{1}{1-\xi} = \frac{1}{|1-\xi|}
e^{2\pi i(\frac{1}{4} -\frac{\theta}{2})} = \frac{1}{|1-\xi|}
e^{-2\pi i\frac{\theta}{2}}i.\end{equation} For $a\coloneqq r(A)$
the maps $df|_{T^*_{[a]}\cM(\Sigma)}$ and
$df^*|_{T_{[a]}\cM(\Sigma)}$ have the same eigenvalues, and we have
$\rk \cN_{[a]} = \rk \cN^*_{[a]}$. Therefore by Proposition
\ref{prop2}, Equation \eqref{SimpleFact} and Remark \ref{RemarkEta}
we get
\begin{align*}\frac{1}{d_c!} r^*\left(\omega_c^{d_c} \cup
\Ch^\bullet(\lambda^c_{-1}{\cM(\Sigma)})^{-1} \right)_A & = \frac{1}{d_c!}\frac{r^*(\omega_c)^{d_c}_A}{\det(1-df)|_{\cN_{[a]}^*}}\\
 & =
\tau_{\Sigma_f}(A)^{\frac{1}{2}} \exp\left(-2\pi i \frac{\tr
\log[df^*|_{T_{[a]}\cM(\Sigma)}]}{2}\right)i^{\rk \cN_{[a]}}\\
& = \tau_{\Sigma_f}(A)^{\frac{1}{2}} e^{\frac{\pi i}{4}\eta(D_A)}.
\end{align*}
In particular, we get
\begin{equation}\label{identificationPart1}
\begin{split}
\frac{1}{d_c!}(\omega_c^{d_c} \cap
\tau_{d_c}(L_\bullet^c(\cO_{\cM(\Sigma)}))) &
=\int_{|\cM(\Sigma)|'_c}\frac{1}{d_c!} \omega_c^{d_c}
\cup\Ch^\bullet(\lambda^c_{-1}{\cM(\Sigma)})^{-1}
\\
& = \int_{A \in \cM(\Sigma_f)'_c} \tau_{\Sigma_f}(A)^{\frac{1}{2}}
e^{\frac{\pi i}{4}\eta(D_A)}
\end{split}
\end{equation}

Observe, that we can rewrite
\[
\tr \log[f^*|_{\cH^{1,0}(\Sigma,\bar\partial)}] - \frac{\rk
[(f^*-\Id)|_{\cH^{1,0}(\Sigma,\bar\partial)}]}{2} = \sum_{0\neq
\tilde\omega_i \in (-\frac{1}{2},\frac{1}{2})} \tilde\omega_i,
\]
where $e^{2\pi i\tilde\omega_i} = \omega_i$, $\tilde\omega_i \in
[-\frac{1}{2},\frac{1}{2})$, are the eigenvalues of the pull-back
$f^*\co \cH^{1,0}(\Sigma,\bar\partial) \to
\cH^{1,0}(\Sigma,\bar\partial)$. By Proposition \ref{prop1} we
therefore have
\[ \det(f)^{-\frac{1}{2}\dim G} = \exp\left(i \pi \dim G \left(\tr \log[f^*|_{\cH^{1,0}(\Sigma,\bar\partial)}]
-
\frac{\rk [(f^* -
\Id)|_{\cH^{1,0}(\Sigma,\bar\partial)}]}{2}\right)\right).
\]
Therefore, it is easy to see from Remark \ref{RemarkEta} that the
leading order term of $\det(f)^{-\frac{1}{2}\zeta}$ is given by
\begin{equation}\label{identificationPart2}
\det(f)^{-\frac{1}{2}\dim G} =e^{-\frac{\pi i}{4}\eta(D_\theta)}.
\end{equation}
Together, \eqref{identificationPart1} and
\eqref{identificationPart2} prove Theorem \ref{ThmIdentification}.
In particular, we have the following.

\begin{thm} \label{ThmStationaryRho} Let each ${\cM(\Sigma_f)'_c}$
be nonempty for every $c\in C$. Then
\begin{equation}\label{EqnStationaryRho}
Z^{(k)}_G({\Sigma_f}) \dot\sim \frac{1}{|Z(G)|}\sum_{c\in C}
\int_{A\in\cM(\Sigma_f)'_c} k^{d_c}e^{2\pi i k \CS_{\Sigma_f}(A)}
\sqrt{\tau_{\Sigma_f}(A)}  e^{2\pi i\frac{\rho_A(\Sigma_f)}{8}},
\end{equation}
and each factor of the integrand gets identified in the leading term
of the Witten-Reshetikhin-Turaev invariants.
\end{thm}

\section{Spectral flow}\label{SpectralFlow}

The spectral flow along a path of formally self-adjoint, elliptic
differential operators $D_t$ is the algebraic intersection number in
$[0,1]\times \R$ of the track of the spectrum
\[\{(t,\lambda) \mid t\in [0,1],\lambda\in \Spec(D_{t})\}\] and the line segment from $(0,-\epsilon)$ to
$(1,-\epsilon)$. We choose the $(-\epsilon,-\epsilon)$--convention,
which makes the spectral flow additive under concatenation of paths
of connections.\footnote{In the
literature one also frequently finds the
$(-\epsilon,\epsilon)$--convention (see for example \cite{freed95, kirk-klassen-ruberman94}), so we need to be careful
when relating to formulas found elsewhere.}

The main statement of this section relating spectral flow, the
Chern-Simons invariant and the $\rho$--invariant for a compact Lie
group seems to be well-known. Since it depends on several
conventions and we have not found a general proof anywhere in the
literature, we decided to provide a proof in this paper in the hope
that it may be a useful reference. With slightly different
conventions, this has been proven in \cite[Section
7]{kirk-klassen-ruberman94} for $\SU(2)$. Even though the main proof
is completely analogous, we give a detailed exposition for the
convenience of the reader.

\subsection*{The dual Coxeter number}

Let $G$ be a simple Lie group of dimension $n$ and rank $r$.
Consider any positive definite normalization $\la \cdot,\cdot
\ra_{\fg}$ of the Killing form on $\fg$. Given a basis
$\{X_i\}_{i=1,\ldots,n}$ of $\fg$ and its dual basis $\{X^i\}$ with
respect to $\la \cdot,\cdot \ra_{\fg}$, the quadratic Casimir is the
element
\[
\Omega = \sum_i X_i \tensor X^i \in \fg \tensor \fg.
\]
As an element of the universal enveloping algebra it commutes with
all elements of $\fg$. The Casimir invariant in the adjoint
representation is given by
\[
\ad_*(\Omega) = \sum_i \ad_*(X_i) \ad_*(X^i) \in \End(\fg).
\]
By Schur's Lemma we know, that it is proportional to the identity
with factor---by definition---the Casimir eigenvalue $C_{\ad}$ in
the adjoint representation with respect to the normalization
$\la\cdot,\cdot\ra_{\fg}$.

Therefore, we have for all $X,Y \in \fg$
\[
\tr(\ad_*(X)\ad_*(Y)) = K\la X,Y\ra_\fg,
\]
where $K$ is determined by
\[ K = K \frac{1}{n} \sum_{i=1}^n\la X_i,X^i \ra_\fg =
\frac{1}{n}\sum_{i=1}^n\tr(\ad_*(X_i)\ad_*(X^i)) = \frac{1}{n}
C_{\ad}\tr(\Id) = C_{\ad}.
\]

The inner product $\la \cdot,\cdot \ra_\fg$ gives rise to the
identification $\fg^* \to \fg$, $\beta \to X_\beta$, where $\beta(X)
= \la X_\beta, X \ra_\fg$ for all $X \in \fg$. We also have an
induced inner product $\la \cdot,\cdot \ra_\fg$ on $\fg^*$ given by
$\la \beta,\gamma \ra_\fg \coloneqq \la X_\beta,X_\gamma \ra_\fg$.
Then we have $C_{\ad} = \la \theta,\theta \ra_\fg \cdot h$ for the
maximal root $\theta$ (see for example \cite[Equation
(1.6.51)]{fuchs1995_AffineLieAlgebrasAndQuantumGroups}), where the
dual Coxeter number $h$ is independent of $\la \cdot,\cdot\ra_\fg$.

Notice that for the inner product on $\SU(n)$ given by $\la
X,Y\ra_{\fsu(n)} = -\tr(XY)$, the maximal root $\tilde\theta$
satisfies $\la\tilde\theta,\tilde\theta\ra_{\fsu(n)} = 2$. We
therefore get
\begin{equation} \label{EqnSUnDualCoxeter}
-\tr (\ad_*(X)\ad_*(Y) ) = -2 n \tr( X Y) \qquad \text{for }
X,Y\in\fsu(n).
\end{equation}

\subsection*{Relating Chern classes via the adjoint representation}

We have seen in Section \ref{RhoInvariant} that with respect to an
ad-invariant metric on $\fg$, we can consider the complexified
adjoint representation
\[
\Ad\co G \to \SU(n) \subset \End(\fg^\C),
\]
and its differential \[\Ad_* \co \fg \to \fsu(n) \subset
\End(\fg^\C).\] It is easy to see, that $C_{\Ad} = C_{\ad}$.

We can define the second Chern form $c_2(B)$ of a connection $B$ in
a principal $G$--bundle $P$ over a $4$--manifold $Z$ by
\[
c_2(B) \coloneqq \la F_B \wedge F_B \ra_\fg,
\]
where $\la\cdot,\cdot\ra_\fg$ is the normalization of the Killing
form on $\fg$ introduced in Section \ref{CS&MS}. This
normalization is given in terms of the Killing form by
\begin{equation}\label{EqnNormalization} \la X,Y\ra_\fg = \frac{1}{16\pi^2 h}
\tr(\Ad_X,\Ad_Y),
\end{equation}
which is shown in \cite[page 242]{freed1998_GeometryOfLoopGroups}
together with a list of the dual Coxeter numbers $h$. Note that in
this normalization $c_2(B)$ represents an integral generator of the
second cohomology.

$\Ad P$, the complexified adjoint bundle of $P$, is a Hermitian
vector bundle, which we view as a principal $\SU(n)$--bundle via its
frame bundle. Therefore, it makes sense to consider the adjoint
bundle $\Ad(\Ad P)$ of $\Ad P$, whose fiber is $\fu(n)$. The
connection $B$ in $P$ induces a connection $\Ad B$ in $\Ad P$ as
follows. Given a section $s \co U \to P$ for $U \subset M$ open,
then $s^*B$ is a $\fg$--valued $1$--forms on $U$. $B$ is uniquely
determined by the family of $1$--forms $B^s_g \coloneqq (sg)^*B$,
where $g \in C^\infty(U;G)$. In this way, $\Ad B$ is determined by
$\{\Ad_* \circ B^s_g\}_{g\in C^\infty(U;G)}$. Similarly we get
$F_{\Ad B} = \Ad F_B$, where $\Ad F_B \in \Omega^2(\Ad (\Ad P))$.

By the previous paragraph we can consider the second Chern form
\[c_2(\Ad B) = \la F_{\Ad B} \wedge F_{\Ad B} \ra_{\fsu(n)} = \la \Ad F_B \wedge \Ad F_B \ra_{\fsu(n)}\] of $\Ad
B$ in $\Ad P$. By Equation \eqref{EqnNormalization} and
\eqref{EqnSUnDualCoxeter} we get
\begin{align*}
c_2(\Ad B) &= \frac{1}{16\pi^2n} \tr( \Ad(\Ad F_B) \wedge \Ad(\Ad
F_B) ) = \frac{2n}{16\pi^2n} \tr (\Ad F_B \wedge \Ad F_B)\\
&= 2n\frac{16\pi^2 h}{16\pi^2n} \la F_B\wedge F_B \ra_\fg = 2 h
c_2(B).
\end{align*}
Notice that $c_1(\Ad B) = 0$, because $\Ad P$ is the
complexification of $\ad P$, and therefore $\ch_2(\Ad B) =
\frac{1}{2}c_1^2(\Ad B) - c_2(\Ad B) = -c_2(\Ad B)$. This gives
\begin{equation}\label{ChernClassRelation}
-\ch_2(\Ad B) = c_2(\Ad B)  = 2 h \, c_2(B)
\end{equation}

\subsection*{The relationship to the \texorpdfstring{$\rho$}{rho}--invariant and the Chern-Simons
function}

We are ultimately interested in the spectral flow $\SF(D_{A_t})$ of
the odd signature operator coupled to a path of connections $A_t$
from the trivial connection $\theta$ to another flat connection $A$.
Since the spectral flow only depends on the endpoints, we will call
this the spectral flow from $\theta$ to $A$
\[\SF(\theta,A) \coloneqq \SF(D_{A_t}).\]

\begin{thm}\label{SFRhoCS} Let $G$ be a simple Lie group and $A$ a flat $G$--connection, then we get \[ \SF(\theta,A) =
-4 h \CS(A)+ \frac{\rho_A(X)}{2} - \frac{\dim G(1+b^1(X))}{2} +
\frac{\dim(H^0(X,d_A)) + \dim(H^1(X,d_A))}{2}.
\]
\end{thm}

We note that this Theorem combined with Theorem \ref{ThmStationaryRho} implies Theorem \ref{ThmStationarySF} from the introduction.

\begin{proof} The proof is analogous to the argument in \cite[Section 7]{kirk-klassen-ruberman94}.
Notice, that because we have $\la X,Y\ra = -\frac{1}{8\pi^2}\tr(XY)$
for $X,Y\in\fsu(n)$, our Chern-Simons function has a different sign
than the Chern-Simons function used for example in
\cite{kirk-klassen-ruberman94,nishi98}. Let $S_B \co \Omega^1 \to
\Omega^0 \oplus \Omega^2_-$ be the self-duality operator on $Z = X
\times I$ defined by $\omega \mapsto (d_B^*\omega, P_-(d_B\omega))$
for a connection $B$ on $Z$, where $P_-$ is the projection to the
anti-self-dual 2--forms. We will use the ``outward normal first''
convention to orient $Z$, so that we do not have to introduce signs
in Stokes' Theorem. Near the boundary we have $S_B \circ \Psi_2 =
\Psi_1 (D_A' + \frac{\partial}{\partial u})$, where $D'_A(a,b) =
(-d_A^*b,
*d_Ab - da_A)$, $A = B |_{\partial Z}$, $\Psi_2(a,b) = a \, du + b$
and $\Psi_1(a,b) = (-a,P_-(b\, du))$. By the Atiyah-Patodi-Singer
index theorem (see also \cite[Theorem
7.1]{kirk-klassen-ruberman94})) we get for the connection $B=A_t$ on
$Z$
\[
\SF(D_{A_t}) = \text{Index} S_B = \int_Z \hat A(Z)
\text{ch}(V_-)\text{ch}(\Ad B) + \frac{1}{2}(\eta(D_{A_1}) + \dim
\ker D_{A_1}) - \frac{1}{2}(\eta(D_{A_0}) + \dim \ker D_{A_0}),
\]
where $\ch (\Ad B)$ is the total Chern character form of the
connection $\Ad B$ in the trivial bundle $Z \times \fg^\C$ induced
by $B$ and $V_-$ is the complex spinor bundle of
$-\frac{1}{2}$--spinors on $Z$, whose rank is 2 for a 4--manifold.
Consider $c_2(B) = \la F_B \wedge F_B \ra$. Then by Stokes' theorem
\[\CS(A_1) - \CS(A_0) = \int_Z c_2(B).\]
By Equation \eqref{ChernClassRelation} we have
\[2h(\CS(A_1)-\CS(A_0)) = 2h \int_Z c_2(B) = - \int_Z \ch(\Ad B).\]  We have
\[\hat A(Z) = 1 + \frac{1}{24} c_2(Z),\] so that the integrand in the index
theorem can be split up
\[
\int_Z \hat A(Z) \text{ch}(V_-)\text{ch}(\Ad B) = \int_Z \left(\hat
A(Z) \text{ch}(V_-)\rk (\Ad B) + 2 \ch_2(\Ad B)\right).
\]
The first contribution can immediately be computed to be zero by
applying the index theorem to a constant path at the trivial
connection. The second contribution is precisely $-4h
(\CS(A_1)-\CS(A_0))$. By definition, the difference of the
$\eta$--invariants $\eta(D_A)-\eta(D_\theta)$ is the
$\rho$--invariant. After identifying the cohomology with the kernel
of the odd signature operator, the theorem follows.
\end{proof}

\appendix
\section{A heuristic discussion of the path integral}
\label{Heuristics}

As a disclaimer, we would like to mention that this appendix reviews
parts of \cite{witten89} and is the only non-rigorous part in the
paper, which we decided to include for motivational purposes. See
Sawon's overview \cite{sawon2003:PertExpCS} on the perturbative expansion of Chern-Simons theory and Rozansky's work \cite{rozansky95} in particular
\cite{rozansky95_LoopExpansion} for a detailed account.

For a function $f\co \R^n \to \R$ with finitely many non-degenerate
critical points and a compactly supported function $\phi \co \R^n
\to \R$, we have the asymptotic behaviour
\[
\int_{\R^n} e^{ik f(x)} \phi(x) \, dx \sim_{k\to \infty}
\left(\frac{2\pi}{k}\right)^{\frac{n}{2}} \sum_{x\in \Crit(f)}
e^{\frac{\pi i}{4} \sign \Hess_x(f)}\frac{e^{ik
f(x)}\phi(x)}{\sqrt{|\det \Hess_x(f)|}}
\]
by the method of stationary phase. We may assume that $\phi(x)=1$
for $x \in \Crit(f)$ and $\phi\equiv 0$ outside of a compact set.
Therefore, we will abuse the notation and eliminate the function
$\phi$ from the formulas entirely. Let $G$ be a simple,
simply-connected, compact Lie group. If $G/Z(G)$ acts freely from
the right on $\R^n$ and $e^{ikf(x)}$ is $G$--invariant, then the
Jacobian $J$ of the $G$ action on $x$ induces the measure
$d[x]=|\det J(x)|dx$ on $\R^n/G$ and we get the leading order
asymptotic behaviour
\begin{equation}\label{stationaryphaseFinite}
\frac{\vol G}{|Z(G)|}\int_{\R^n/G} \!\!\!\!e^{ik f(x)} \, d[x] \sim_{k\to \infty}
\left(\frac{2\pi}{k}\right)^{\frac{n-\dim G}{2}} \!\!\!\!\!\!\!\!\sum_{[x]\in
\Crit(f)/G} \!\!\!\!\!e^{\frac{\pi i}{4} \sign \Hess_x(f)}\frac{e^{ik
f(x)}}{\sqrt{|\det \Hess_x(f)|}} \frac{|\det J(x)| \vol G}{|Z(G)|}.
\end{equation}
Also see \cite[Section 2.2]{witten91} and \cite[Section
2.2]{rozansky95_LoopExpansion} for the appearance of the factor
$\frac{1}{|Z(G)|}$.

According to Witten \cite{witten89}, the invariants $Z^{(k)}_G(X)$
can be written as the path integral characterizing the Chern-Simons
theory
\[Z^{(k)}_G(X) = \int_{A\in \cA}e^{2\pi ik \CS(A)}\, dA,\] where we have identified
$\cA=\Omega^1(X,\fg)$\kommentar{where $\cB =
\cA/\cG$ is the moduli space of connections, and we have identified
$\cA=\Omega^1(X,\fg)$, $\cG=C^\infty(X,G)$}. Even though the
right-hand side is not mathematically rigorous, we would like to
formally apply the above stationary phase approximation to this path
integral. This procedure in quantum field theory is known as the
Faddeev-Popov method (see for example
\cite{reshetikhin2010,reshetikhin2010b} for more information). $\cG=C^\infty(X,G)$ acts on $\cA$. It
can easily be seen that $|Z(\cG)|=|Z(G)|$, and we need to ignore $\vol \cG$. In our case, $\det D$ is
the zeta-regularized determinant of a formally self-adjoint elliptic
differential operator $D$
\[
\det D = e^{-\zeta_k'(0)}, \text{ where } \zeta(s) =
\sum_{\lambda_j\neq 0} \lambda_j^{-s},
\]
where $\lambda_j$ are the eigenvalues of $D$. The differential of
the $\cG$ action on $A$ can be seen to be $d_A$. Observe that
\[\|d_A \phi\|^2 = \la\Delta_A^{(0)} \phi,\phi\ra_{L^2} = \lambda
\|\phi\|^2,\] where the $L^2$ inner product on $\Omega^k(X;\fg)$ is
given by
\[
 \la a,b \ra_{L^2} = \int_X \la a \wedge * b \ra,
\]
$\Delta^{(k)}_A$ is the twisted Laplacian on $\Omega^k(X,\fg)$ and
$\phi$ is an eigenvector of $\Delta_A^{(0)}$ with (positive) eigenvalue
$\lambda$. Therefore we have \begin{equation}\label{Jacobian}|\det J
(A)| = \sqrt {\det \Delta^{(0)}_A},\end{equation} which is the
Faddeev-Popov determinant in disguise.

On a finite-dimensional Riemannian manifold we have
\[\Hess_x(f)(X,Y) = \la \nabla_X \grad f, Y \ra\]
for a critical point $x$ of a Morse-function $f$, where $\nabla$ is
the Levi-Civita connection. We can view the $L^2$ inner product on
the space of connections $\cA$ as a metric on $\cA$. We can use it
to identify vectors and covectors of $T_{[A]}(\cB) = \coker d_A\cong
\ker d_A^*$. With respect to this the linearization of $\CS\co \cB \to
\R/\Z$ is given by the gradient $\grad \CS|_A=*F_A\co \ker d_A^* \to
\ker d_A^*$. Consider now the odd signature operator coupled to a
connection $A$, as defined in \eqref{OddSignature} Notice that
$D_A^2 = \Delta_A^{(0)} \oplus \Delta_A^{(1)}$ and therefore $(\det
D_A)^2 = \det \Delta_A^{(0)} \det \Delta_A^{(1)}$. Let $A$ be flat,
then we have under the decomposition $\Omega^1(X;\fg) = \im d_A
\oplus \ker d_A^*$
\[D_A = H_A \oplus S_A\] where
\begin{align*}
S_A\co \Omega^{0}(X;\fg) \oplus \im d_A&\longrightarrow \Omega^{0}(X;\fg) \oplus \im d_A\\
(\alpha,\beta) &\longmapsto (d_A^* \beta, d_A \alpha)
\end{align*}
has symmetric spectrum and satisfies $|\det S_A| = \det
\Delta^{(0)}_A$, while $H_A = \proj_{\ker d_A^*}
*d_A \co \ker d_A^* \to \ker d_A^*$ is the linearization of
$\grad_A\CS$ satisfying $\la H_A(a),b\ra = \Hess_A\CS(a,b).$
Therefore we have \begin{equation}\label{HessCS} |\det \Hess_A\CS| =
|\det H_A| = \frac{|\det D_A|}{|\det S_A|} = \frac{|\det D_A|}{\det
\Delta^{(0)}_A}.
\end{equation}
 The analytic torsion
\[
T_X(A)\coloneqq \prod_k (\det\Delta_A^{(k)})^{(-1)^{k+1}k/2}
\]
is an invariant of Riemannian manifolds defined by Ray and Singer,
which proved to be equal to the Reidemeister torsion $\tau_X(A)$ by
work of Cheeger and M\"uller after choosing the volume form on
cohomology induced by the metric on the manifold. Since
$\det\Delta_A^{(k)} = \det \Delta_A^{(3-k)}$ by Poincar\'e duality,
we deduce from Equations \eqref{Jacobian} and \eqref{HessCS}
\begin{equation}\label{RaySingerTorsion} \sqrt{\tau_X(A)} = (\det
\Delta_A^{(0)})^{3/4}(\det \Delta_A^{(1)})^{-1/4} = \frac{\det
\Delta^{(0)}_A}{\sqrt{|\det D_A|}} = \frac{|\det J(A)|}{\sqrt{|\det
\Hess_A\CS|}}.
\end{equation}

Let us turn to the analogue of the signature. In finite dimensions
we have for a path $x_t$ between two nondegenerate critical points
$x_0$ and $x_1$ we get
\[\sign(\Hess_{x_1}(f))- \sign(\Hess_{x_0}(f)) = 2 \SF(\nabla \grad_{x_t} f),\]
where the spectral flow $\SF$ is defined in Section \ref{SpectralFlow}. Therefore,
instead of the signature of the Hessian, we can use twice the
spectral flow of $H_{A_t}$ for a path of connections $A_t$ from the
trivial connection $\theta$ to some flat connection $A=A_1$. Since
$S_{A_t}$ has symmetric spectrum and $D_{A_t}-H_{A_t}$ is a compact
operator for all $t$, we can use $2\SF(D_{A_t})=2\SF(H_{A_t})$. Keep
in mind, that this procedure neglects the signature at the trivial
connection. Note, that this is the idea behind the gauge-theoretic
version of Casson's invariant for homology $3$--spheres by Taubes
\cite{taubes90} and its generalizations. This turned out to be the
perfect approach for the Casson invariant, because we needed an
integer-valued analogue to the signature. The case of the
Witten-Reshetikhin-Turaev invariants allows for an alternative
approach.

We can consider the $\eta$--invariant defined in
\eqref{EqnEtaInvariant} as a generalized signature. This has the
immediate merit of being defined for every connection, but it is
metric-dependent and not necessarily an integer. Since the
$\rho$--invariant defined in \eqref{EqnRhoInvariant} is independent
of the metric, we will choose it as a generalized signature, keeping
in mind that we introduced $\eta(D_\theta)$. By following the
arguments in \cite{witten89}, $\eta(D_\theta)$ can be altered into a
prefactor, which is a topological invariant of a framed, oriented
manifold. It was observed in \cite{freed-gompf91} that this
prefactor vanishes for the (canonical) Atiyah 2--framing. For
further details we refer to \cite[Section 2]{witten89} and
\cite[Section 1]{freed-gompf91}.

It has been mentioned by Jeffrey \cite[Section 5.2.2]{jeffrey92}
that Reidemeister torsion can be used as a density, thereby
extending the above use of Reidemeister torsion in the formal
application of the stationary phase method to degenerate critical
points. We need this idea to allow for critical components of
positive dimension. To this end we
identify $T_A \cM(X)$ with $H^1(X,d_A)$ and $H^2(X,d_A)$
with $(H^1(X,d_A))^*$ using Poincar\'e duality. Note that Poincar\'e
duality depends on a choice of inner product on $\fg$, which is
possibly a multiple of our original choice of inner product on
$\fg$. Furthermore, we need to choose a suitable volume form or
density on $H^0(X,d_A)$ and $(H^3(X,d_A))^*$, for example we might
take the one induced by the inner product on $\fg$ as suggested in \cite[Section 5.2.2]{jeffrey92} or further normalized as suggested on \cite[page 284]{rozansky95}.

Since we allow higher-dimensional components in the moduli space of
flat connections and the stationary phase approximation in finite
dimensions \eqref{stationaryphaseFinite} includes the factor
$k^{-\frac{n-\dim G}{2}}$, we have to shift our
result by the factor $k^{d_c}$, where $d_c$ is half the real
dimension of the critical component $\cM(X)_c$ minus the dimension of the stabilizer of some generic $[A]\in \cM_c$ under $\cG$. Since the tangent space to the stabilizer is isomorphic to $H^0(X,d_A)$, we expect \[d_c =
\frac{1}{2} \max_{A \in\cM(X)_{c}} \left( \dim(H^1(X,d_A)) -
\dim(H^0(X,d_A)) \right),\] which is also known as the growth rate conjecture (see
\cite[Lemma 7.2]{andersen95} for evidence). By the same argument we
may like to introduce the factor $\frac{1}{(2\pi)^{d_c}}$. For
similar reasons Rozansky \cite[Equation
(2.33)]{rozansky95_LoopExpansion} includes such a factor. We will
simply set every factor of the form $K^{d_c}$ for a constant $K>0$
to 1, because a change of normalization for Poincar\'e
duality (used to treat Reidemeister torsion as a density) by a
factor $K$ results in the factor $K^{-d_c}$ in the stationary phase
approximation. The other factors, which only depend on $n$ and the
dimension of $G$, we need to omit, because both $\cA$ and $\cG$ are
infinite-dimensional.

If we therefore replace the Signature of the Hessian by the
$\rho$--invariant \eqref{EqnRhoInvariant}, replace the rest via
Equation \eqref{RaySingerTorsion} and normalize Poincar\'e duality
appropriately (independently of $X$ and $G$), the following
conjecture is justified.

\begin{conj} Let $G$ be a simple, simply-connected, compact Lie group. Let $X$ be a closed 3--manifold and $C$ the set of all connected components of $\cM(X)$. Then, in the Atiyah
2--framing, the leading order asymptotic behavior of $Z_G^{(k)}(X)$
in the limit $k\to \infty$ is given by
\begin{equation}\label{stationaryphaseRho}
\begin{split}
Z_{G}^{(k)}(X) \dot\sim \sum_{c\in C} \frac{1}{|Z(G)|}\int_{A\in
\cM(X)_c} &\sqrt{\tau_{X}(A)} e^{2\pi i \CS_{X}(A)k} e^{\frac{\pi
i}{4} \rho_A(X)} k^{d_c}.
\end{split}
\end{equation}
\end{conj}

Theorem \ref{SFRhoCS} immediately yields the more familiar version
\eqref{stationaryphaseSF} of \eqref{stationaryphaseRho} stated in
the introduction.\kommentar{\begin{conj} Let $G$ be a simple,
simply-connected, compact Lie group, $X$ a closed 3--manifold and
$C$ the set of all connected components of $\cM(X)$. Then, in the
Atiyah 2--framing, the leading order asymptotic behavior of
$Z_G^{(k)}(X)$ in the limit $k\to \infty$ is given by
\begin{equation*}
\begin{split}
Z_{G}^{(k)}(X) \dot\sim \sum_{c\in C} \frac{1}{|Z(G)|}e^{\pi i \dim
G(1+b^1(X))/4} \int_{A\in \cM(X)_{c}} &\sqrt{\tau_X(A)} e^{2\pi i
\CS_X(A)(k+h)}\\& e^{2\pi i
\left(\SF(\theta,A)/4-(\dim(H^0(X,d_a))+\dim(H^1(X,d_a)))/8\right)}k^{d_c}.
\end{split}
\end{equation*}
\end{conj}} Observe that the Chern-Simons invariant is constant on connected
components of flat connections, we could therefore put it in front
of the integral. On reducible subsets it will be necessary to
interpret these conjectures in a suitable way, however we do not
consider the reducible case in this paper.

\kommentar{It is more appropriate to view and formulate the
heuristics in relation to the asymptotic expansion given by
Conjecture \eqref{ConjAsymptoticExpansion}. Even though the above
conjectures can only be understood as an asymptotic estimate, that
is, the ratio of both sides tends to 1 as $k\to \infty$, we would
like understand it as a part of the conjectured asymptotic
expansion, which will incorporate the higher loop contributions.

\begin{conj} Let $X$ be a
closed oriented 3--manifold. In the Atiyah 2--framing, the
asymptotic behavior of $Z_G^{(k)}(X)$ in the limit $k\to \infty$ is
given by
\begin{equation}
\begin{split}
Z_G^{(k)}(X)\sim \sum_{c\in C} \frac{1}{|Z(G)|}\int_{A\in \cM(X)_c}
&\sqrt{\tau_{\Sigma_f}(A)} e^{2\pi i \CS_{\Sigma_f}(A)k}
e^{\frac{\pi i}{4} \rho_A(\Sigma_f)} k^{d_c} \exp
\left(\sum_{l=1}^\infty
\frac{c^lk^{-l}}{(2l)!(3l)!}\sum_{e(\Gamma)=-l}
\frac{Z_\Gamma(X,A)}{|\Aut(\Gamma)|}\right),
\end{split}
\end{equation}
where $C$ be the set of all connected components of $\cM(X)$, $c$ is
some scalar, $\Gamma$ is trivalent graph with Euler number
$e(\Gamma)$ and $\Aut(\Gamma)$ is the order of the automorphism
group of $\Gamma$.
\end{conj}}

\section{Review of the Lefschetz-Riemann-Roch theorem for singular varieties.}\label{LRR}

We shall very quickly review the
Lefschetz-Riemann-Roch theorem for singular varieties due to P. Baum,
W. Fulton,
R. MacPherson
and G. Quart (see \cite{baum-fulton-macpherson1979_Riemann-Roch} for a proof of the Riemann-Roch theorem
and the general theory and \cite{baum-fulton-quart_LefschetzRiemann-Roch} for a proof of the Lefschetz-Riemann-Roch
theorem.) We will only state their theorems in the generalities we need.

Let $X$ be a complex quasi-projective algebraic variety. Consider the
Grothendieck group $K_{\alg}^0 (X)$ of algebraic vector bundles (i.e. locally
free sheaves) on $X$. This is a ring-valued contravariant functor. Let
$K^{\alg}_0(X)$ be the Grothendieck group of coherent sheaves of ${\mathcal
O}_X$ modules on $X$. This is a covariant functor for proper morphism: If $f\co X
\rightarrow Y$ is a proper morphism, then \[f_*\co K^{\alg}_0(X) \rightarrow
K^{\alg}_0(Y),\] is defined by setting $f_*[{\mathcal F}] = \sum (-1)^i
[R^if_*{\mathcal F}].$

For a topological space $X$ one considers the Grothendieck group
$K_{\topo}^0(X)$ of topological vector bundles on $X$, so $K_{\topo}^0(X)$ is a
ring-valued contravariant functor. Let $K^{\topo}_0(X)$ be the Grothendieck
group of complexes of vector bundles on ${\C}^N$ exact off $X$ for some
closed embedding of X in ${\C}^N$. (One is making Alexander duality a
definition here.)

For any complex algebraic variety $X$ there is a natural ring homomorphism
\[\alpha^\bullet \co K_{\alg}^0(X) \rightarrow K_{\topo}^0(X),\] which is a natural
transformation of contravariant functors. Suppose $X$ is a closed algebraic
subset of a variety $Y$. Let $K^{\alg}_X(Y)$ be the Grothendieck group of
complexes of algebraic vector bundles on $Y$ which are exact off $X$. There is a
natural homology map \[h \co K^{\alg}_X(Y) \rightarrow K^{\alg}_0(X),\] given by
\[h([E_\bullet]) = \sum (-1)^i [H_i(E_\bullet)]\] where the $H_i(E_\bullet)$ are the homology
sheaves of the complex $E_\bullet$ of locally free sheaves on $Y$. The map $h$ is
an isomorphism. (See \cite[Appendix 2]{baum-fulton-macpherson1979_Riemann-Roch}.) Suppose $X$ is a closed subspace of
$Y$, where $Y$ is a $C^\infty$-manifold. When we have a closed embedding of
$C^\infty$-manifolds $Y \hookrightarrow {\C}^N$ and the normal bundle of $Y$
in ${\C}^N$ has a complex structure, we get the Thom-Gysin isomorphism \[h \co
K^{\topo}_X(Y) \rightarrow K^{\topo}_X({\C}^N) = K^{\topo}_0(X).\] Again for
any closed subset $X$ of an algebraic variety $Y$ we have homomorphism of
abelian groups $\alpha^{\bullet} \co K^{\alg}_X(Y) \rightarrow K^{\topo}_X(Y)$. We
shall now describe the key construction in the formulation of the Riemann-Roch
theorem in \cite{baum-fulton-macpherson1979_Riemann-Roch}. There is a homomorphism \[\alpha_\bullet \co K^{\alg}_0(X)
\rightarrow K^{\topo}_0(X)\] of abelian groups, which is covariant for proper
morphisms defined the following way: Choose an embedding of $X$ in a nonsingular
variety $Y$, then $\alpha_\bullet$ is the composition \[ \alpha_\bullet \co
K^{\alg}_0(X) \stackrel{h^{-1}}{\rightarrow} K^{\alg}_X(Y)
\stackrel{\alpha^\bullet}{\rightarrow} K^{\topo}_X(Y) \stackrel{h}{\rightarrow}
K^{\topo}_0(X).\] With this setup at hand the main theorem in \cite{baum-fulton-macpherson1979_Riemann-Roch} is
stated. (See \cite[pages 174--75]{baum-fulton-macpherson1979_Riemann-Roch}.) We are however only interested in the
weaker version of this theorem where topological K-theory is replaces by
ordinary homology theory with rational coefficients. Let $H^\bullet(X)$ be
ordinary singular cohomology with rational coefficients. If $X$ is closed in
$Y$, let $H_X^\bullet(Y) = H^\bullet(Y,Y-X)$. Again, we use the Alexander duality to
define the homology groups \[H_i(X) = H_X^{2n-i}({\C}^N).\]

Let $\Ch^\bullet \co K_{\topo}^0(X) \rightarrow H^\bullet(X)$ be the usual Chern
character and let $\Ch^\bullet \co K^{\topo}_X(Y) \rightarrow H^\bullet_X(Y)$ be the
canonical extension of the Chern character. We can define the homological Chern
character by embedding $X$ in some ${\C}^N$ and then define $\Ch_\bullet$ to
be the composition \[\Ch_\bullet \co K^{\topo}_0(X) = K^{\topo}_X({\C}^N)
\stackrel{\Ch^\bullet}{\rightarrow} H^\bullet_X({\C}^N) = H_\bullet(X).\]
We shall also use the notation $\Ch^\bullet$ for the composition
$ \Ch^\bullet \alpha^\bullet \co K_{\alg}^0 \rightarrow H^\bullet$ and we define
$\tau_\bullet = \Ch_\bullet \alpha_\bullet \co K^{\alg}_0 \rightarrow H_\bullet$. The
Riemann-Roch theorem for singular varieties can now be formulated as follows.

\begin{thm}[{\cite[page 180]{baum-fulton-macpherson1979_Riemann-Roch}}]
The mapping
\[\tau_\bullet \co
K^{\alg}_0(X) \rightarrow H_\bullet(X)\]
is covariant for proper morphisms,
compatible with cap products, cartesian products and restrictions to open
subvarieties. If $X$ is non-singular
\[\tau_\bullet[{\mathcal O}_X] = \Td(T_X) \cap
[X].\]
\end{thm}

We remark that for a projective variety
$$
\tau_\bullet[{\mathcal O}_X] =
[X]
$$
modulo lover degree terms. This follows from the lemma on page 129 of \cite{baum-fulton-macpherson1975:RiemannRoch} and part (6) of the Riemann-Roch Theorem of \cite{fulton-gillet1983:RiemannRoch}.

Let us now press on with the Lefschetz-Riemann-Roch
theorem. An equivariant variety $X$ will be defined to be a quasi-projective
algebraic variety with an automorphism $x\co X \rightarrow X$, such that $x^m =
\Id$. The fixed point subvariety we will denote $|X|$. We will for the rest of
this section assume that the varieties we are considering are equivariant,
unless otherwise stated.

An equivariant sheaf on $X$ is a coherent sheaf ${\mathcal F}$ of ${\mathcal
O}_X$ modules together with a homomorphism of sheaves \[\phi_{\mathcal F} \co
x^*{\mathcal F} \rightarrow {\mathcal F}.\] Let $K^{\eq}_0(X)$ be the
Grothendieck group of all equivariant sheaves on $X$ and let $K_{\eq}^0(X)$ be
the Grothendieck group of all equivariant locally free sheaves on $X$.

If the automorphism $x$ is the identity, then any equivariant sheaf ${\mathcal
F}$ on $X$ breaks up into a finite direct sum of sheaves ${\mathcal F}_a$, $a\in
{\C}$, where ${\mathcal F}_a$ is the generalized a-eigen-sheaf for
$\phi_{\mathcal F}$. This gives maps \begin{equation} K_{\eq}^0(X) \rightarrow
K_{\alg}^0(X) \otimes {\Z}[{\C}] \rightarrow K_{\alg}^0(X)\otimes {\C}\label{1} \end{equation} and \begin{equation}K^{\eq}_0(X) \rightarrow
K^{\alg}_0(X) \otimes {\Z}[{\C}] \rightarrow K^{\alg}_0(X)\otimes {\C},\label{2} \end{equation} by mapping $[{\mathcal F}]$ to $\sum [{\mathcal
F}_a]\otimes a$ followed by the natural trace map $\tr \co{\Z}[{\C}] \rightarrow {\C}.$ In particular, if we in the general case compose the
homomorphism in \eqref{1} with the homomorphism coming from the inclusion map
$|X|\hookrightarrow X$, we get a natural homomorphism \[L^\bullet \co K_{\eq}^0(X)
\rightarrow K_{\alg}^0(|X|)\otimes {\C}.\]

If $V$ is a component of $|X|$ and $X$ is non-singular in a neighborhood of $V$,
then $V$ is also non-singular, and the conormal sheaf ${\mathcal N}$ to $V$ in
$X$ is an equivariant locally free sheaf on $V$. Then $\lambda_{-1}({\mathcal
N}) = \sum (-1)^i[\Lambda^i {\mathcal N}]$ determines an element in
$K_{\eq}^0(V),$ which in turn under \eqref{1} maps to the element, say
\[\lambda_{-1}^V X \in K_{\alg}^0(V)\otimes {\C}.\] This element is clearly
invertible in $ K_{\alg}^0(V)\otimes {\C}.$

In order to state the Lefschetz-Riemann-Roch theorem, we need to discuss
relative equivariant K-theory. Let $X$ be a closed equivariant subvariety of
$Y$. Define $K_X^{\eq}(Y)$ to be the Grothendieck group of equivariant complexes
on $Y$ which are exact off $X$. Suppose now that $Y$ is non-singular and that $j
\co X \rightarrow Y$ is the inclusion map and that $|X|$ is projective. We have
the homology isomorphism \[h \co K_X^{\eq}(Y) \rightarrow K_0^{\eq}(X)\] defined
the same way as in the non-equivariant case. We also define the modified
homology map \[\tilde{h} \co K_{|X|}^{\eq}(|Y|)\otimes {\C} \rightarrow
K_0^{\alg}(|X|)\otimes {\C}\] by the formula \[\tilde{h}(\xi) =
|j|^*(\lambda^{|Y|}_{-1}Y)^{-1} \cap h(\xi).\] There is a natural homomorphism
(in the case $|X|$ is projective) \[L \co K^{\eq}_X(Y) \rightarrow
K_{|X|}^{\eq}(|Y|)\otimes {\C} \] for $X$ closed in $Y$, given by the
pull-back homomorphism induced by the inclusion of $|Y|$ in $Y$. Now define a
homomorphism

\[L_\bullet \co K_0^{\eq}(X) \rightarrow K_0^{\alg}(|X|)\otimes {\C}\] to be the
composition \[L_\bullet \co K_0^{\eq}(X) \stackrel{h^{-1}}{\rightarrow} K_X^{\eq}(Y)
\stackrel{L}{\rightarrow} K_{|X|}^{\eq}(|Y|)\otimes {\C}
\stackrel{\tilde{h}}{\rightarrow} K_0^{\alg}(|X|)\otimes {\C}\] for some
closed embedding of $X$ in a non-singular $Y$.

The Lefschetz-Riemann-Roch theorem for singular varieties can now be stated as
follows.

\begin{thm}[Baum, Fulton \& Quart] \label{LRRT} The homomorphism \[L_\bullet \co
K_0^{\eq}(X) \rightarrow K_0^{\alg}(|X|)\otimes {\C}\] is independent of the
embedding of $X$ in a non-singular $Y$ and is compatible with cap-products,
cartesian products, restrictions to open equivariant sub-varieties. Moreover
$L_\bullet$ is covariant for proper morphisms and if $X$ is non-singular around a
component $V$ of $|X|$, then \[L^V_\bullet [{\mathcal O}_X] = (\lambda^V_{-1}
X)^{-1} \cap [{\mathcal O}_V] \in K_{\alg}^0(V)\otimes {\C}.\] \end{thm}

Suppose now that $X$ is an equivariant projective algebraic variety and that
${\mathcal E}$ is an equivariant locally free sheaf on $X$. By pushing forward
to a point and combining the two theorems above we get the following
Lefschetz-Riemann-Roch formula due to Baum, Fulton, MacPherson and Quart
\[ \sum
(-1)^i \tr (x \co H^i(X,{\mathcal E}) \rightarrow H^i(X,{\mathcal E})) =
\Ch^\bullet(L^\bullet({\mathcal E})) \cap \tau_\bullet L_\bullet({\mathcal O}_X). \]
Here
$\cap \co H_\bullet(|X|,{\C})\otimes H^\bullet(|X|,{\C}) \rightarrow {\C}$
is the cap product pairing between cohomology and homology. Let $C$ be the
finite set which indexes connected components of $|X|$, i.e.
\[|X| =
\coprod_{c\in C} |X|_c.\]
Denote $\dim |X|_c = n_c$. Suppose furthermore that
${\mathcal E} = {\mathcal L}^k$ where ${\mathcal L}$ is an equivariant line
bundle over $X$. Say $c_1({\mathcal L}) = \alpha$ and denote $\alpha|_{|X|_c} =
\alpha_c$. Let $a_c \in {\C}$ be such that
\[
\Ch^\bullet(L_c^\bullet({\mathcal L}^k)) = \exp(k \alpha_c) \otimes a_c^k.\]
The
Lefschetz-Riemann-Roch formula then reads
\begin{align*} \sum (-1)^i \tr (x \co
H^i(X,{\mathcal L}^k) \rightarrow H^i(X,{\mathcal L}^k)) & =  \sum_{c\in
C}a_c^k\exp(k\alpha_c) \cap \tau_\bullet(L_\bullet^c({\mathcal O}_X)) \nonumber \\ & =
\sum_{c\in C} a_c^k \left(\sum_{i=0}^{n_c}\frac{1}{i!} (\alpha_c)^i \cap
\tau_i(L^c_\bullet({\mathcal O}_X)) k^i\right) \end{align*}
If $|X|_c$
is contained in the non-singular part of $X$, we get that
\begin{align*}
\exp(k \alpha_c) \cap \tau_\bullet(L_\bullet^c({\mathcal O}_X)) &= \exp(k
\alpha_c) \cap \Ch^\bullet(\lambda_{-1}^cX)^{-1} \cap \tau_\bullet ([{\mathcal
O}_{|X|_c}])\\ & =  ( \exp(k \alpha_c) \cup \Ch^\bullet(\lambda_{-1}^cX)^{-1}
\cup \Td(T_{|X|_c})) \cap [|X|_c]. \end{align*}

\bibliographystyle{himpel_gtart}
\bibliography{references}

\end{document}